\pdfoutput=1
\documentclass[11pt, reqno]{amsart}
\usepackage{amsmath,amsthm,amsfonts,amssymb,mathrsfs,bm,graphicx,stmaryrd}
\usepackage{epstopdf}
\usepackage{mathtools}
\usepackage{dsfont}
\usepackage{multicol}
\usepackage[colorlinks=true,linkcolor=blue,citecolor=blue,breaklinks]{hyperref}
\usepackage{url}

\usepackage{breakurl}
\usepackage{bbm}
\usepackage{subcaption}
\usepackage{enumerate}
\usepackage{longtable}

\newcommand{\RR}{\mathbb{R}}
\newcommand{\CC}{\mathbb{C}}

\newcommand{\DD}{\mathbb{D}}

\newcommand{\NN}{\mathbb{N}}

\newcommand{\ZZ}{\mathbb{Z}}

\newcommand{\giv}{\,|\,}

\usepackage[letterpaper,hmargin=1.0in,vmargin=1.0in]{geometry}
\parskip 	\smallskipamount

\newtheorem{theorem}{Theorem}
\newtheorem{lemma}[theorem]{Lemma}

\newtheorem{definition}[theorem]{Definition}
\newtheorem{corollary}[theorem]{Corollary}
\newtheorem{proposition}[theorem]{Proposition}

\newtheorem{remark}[theorem]{Remark}

\newcommand{\Cov}{\mathrm{Cov}}

\usepackage[backend=biber,style=alphabetic,doi=false,maxalphanames=10,maxnames=50]{biblatex}
\AtBeginBibliography{\small}

\begin{document}
\title[]{Convergence of the subcritical $\gamma$-LQG metric to the critical $2$-LQG metric as $\gamma \to 2$}

\author[]{Konstantinos Kavvadias}
\address{Konstantinos Kavvadias, Courant Institute of Mathematical Sciences, New York University, New York, NY, USA}
\email{kk6501@nyu.edu}
\date{}

\begin{abstract}
We show that the $\gamma$-LQG metric (appropriately re-normalized) for $\gamma \in (0,2)$ converges in law to the critical $2$-LQG metric (appropriately re-normalized) with respect to the local uniform topology of continuous functions on $\CC \times \CC$. Combining with the results in \cite{kavvadias2025continuity}, we obtain that the law of the $\gamma$-LQG metric (appropriately re-normalized) is continuous in $\gamma \in [0,2]$ with respect to the local uniform topology of continuous functions on $\CC \times \CC$, where the $0$-LQG metric corresponds to the Euclidean metric on $\CC$.
\end{abstract}

\maketitle

\tableofcontents

\section{Introduction}
\label{sec:intro}

\subsection{Overview}
\label{subsec:overview}
Liouville quantum gravity (LQG) surfaces were first introduced in the physics literature by Polyakov \cite{polyakov1981quantum,polyakov1989quantum}. More precisely, we fix $\gamma \in (0,2)$, an open and connected set $U \subseteq \CC$, and let $h$ be the Gaussian free field (GFF) on $U$. Then the $\gamma$-\emph{Liouville quantum gravity} (LQG) surface parameterized by $(U,h)$ is formally the two-dimensional Riemannian manifold with metric tensor
\begin{equation}\label{eq:lqg_metric_tensor}
e^{\gamma h}(d x^2 + d y^2),
\end{equation}
where $dx^2 + dy^2$ denotes the Euclidean metric tensor.

Note that the definition in \eqref{eq:lqg_metric_tensor} does not make literal sense since $h$ is not well-defined as a function taking pointwise values but rather as a distribution (generalized function). However, it was shown in \cite{GM21} (see also \cite{DDDF20,DFGPS20,MS20,MS16,MS21}) that we can construct a random metric on $U$ associated with \eqref{eq:lqg_metric_tensor} using a regularization procedure, for all $\gamma \in (0,2)$. Moreover, it has been shown that \eqref{eq:lqg_metric_tensor} can be associated with a random volume form $e^{\gamma h(z)} dz$ where $dz$ denotes the Lebesgue measure via regularization procedures (see \cite{Kah85,DS09,RV13}).

\textbf{Liouville first passage percolation.} As mentioned above, to construct the LQG metric, we use a regularization procedure. In particular, we use a family of functions which approximate $h$. For $s>0$ and $z,w \in \CC$, we let
\begin{equation*}
    p_s(z,w) = \frac{1}{2\pi s} \exp\left(-\frac{|z-w|^2}{2s}\right)
\end{equation*}
be the transition kernel of the two-dimensional standard Brownian motion. For $\varepsilon > 0$, we consider a mollified version of the GFF by 
\begin{equation}\label{eq:gff_mollification}
h_{\varepsilon}^{\star}(z):=(h \star p_{\varepsilon^2 / 2})(z) = \int_U h(w) p_{\varepsilon^2 / 2}(z,w) dw = (h,p_{\varepsilon^2 / 2}(z,\cdot)) \quad \text{for all} \quad z \in U,  
\end{equation}
where the integral is interpreted in the distributional sense. Also, for all $\xi>0$, we define
\begin{equation}\label{eq:lfpp_metric}
D_h^{\varepsilon}(z,w):=\inf_P \int_0^1 e^{\xi h_{\varepsilon}^{\star}(P(t))} |P'(t)| dt \quad \text{for all} \quad z,w \in \CC,
\end{equation}
where the infimum is over all piecewise continuously differentiable paths $P:[0,1] \rightarrow \CC$ from $z$ to $w$. The metrics $D_h^{\varepsilon}$ are called $\varepsilon$-\emph{Liouville first passage percolation} (LFPP).

To extract a non-trivial limit of the metrics $D_h^{\varepsilon}$, we need to re-normalize. We define the re-normalization factor $\alpha_{\varepsilon}$ to be given by the median of
\begin{equation}\label{eq:lfpp_renormalization}
\inf_P \left \{\int_0^1 e^{\xi h_{\varepsilon}^{\star}(P(t))} |P'(t)| dt : P \,\, \text{is a left-right crossing of} \,\, [0,1]^2 \right \}
\end{equation}
in the case that $h$ is a whole-plane GFF normalized such that its average on the unit circle $\partial \DD$ is equal to zero, where a left-right crossing of $[0,1]^2$ is a piecewise continuously differentiable path in $[0,1]^2$ joining the left and right boundaries of $[0,1]^2$.

It was shown in \cite[Proposition~1.1]{ding2023tightness} that for each $\xi>0$, there exists $Q = Q(\xi)>0$ such that
\begin{equation}\label{eq:lfpp_normalizing_factor}
    \alpha_{\varepsilon} = \varepsilon^{1-\xi Q + o_{\varepsilon}(1)}, \quad \text{as} \quad \varepsilon \to 0.
\end{equation}
Moreover, it was shown in \cite[Proposition~1.1]{ding2023tightness} that $Q$ is a non-increasing function of $\xi$.

We define the \emph{critical value} for the parameter $\xi$ by 
\begin{equation}\label{eq:critical_value}
\xi_c:=\inf\{\xi >0 : Q(\xi) = 2\}.
\end{equation}
It also follows from \cite[Proposition~1.1]{ding2023tightness} that $\xi_c$ is the unique value of $\xi$ such that $Q(\xi) = 2$ and from \cite[Theorem~2.3]{gwynne2019bounds} that $\xi_c \in [0.4135 , 0.4189]$. Moreover, we have that $Q>2$ for $\xi < \xi_c$, and $Q \in (0,2)$ for $\xi > \xi_c$.
\begin{definition}\label{def:regimes_of_lfpp}
We refer to LFPP with $\xi < \xi_c, \xi = \xi_c$ and $\xi > \xi_c$ as the \textbf{subcritical},\textbf{critical}, and \textbf{supercritical} phases, respectively.   
\end{definition}

\textbf{Subcritical, critical, and supercritical regimes.} In the subcritical regime, one of the main results in \cite{GM21} in the case that $h$ is a whole-plane GFF is the following.
\begin{theorem}(\cite[Theorem~1.1]{GM21}) \label{thm:uniqueness_of_subcritical_lqg_metrics}
The random metrics $\alpha_{\varepsilon}^{-1} D_h^{\varepsilon}$ converge in probability as $\varepsilon \to 0$ with respect to the local unfirom topology on $\CC \times \CC$ to a random metric in $\CC$ which is almost surely determined by $h$.
\end{theorem}
Moreover, it was shown in \cite{GM21} that the limit in Theorem~\ref{thm:uniqueness_of_subcritical_lqg_metrics} satisfies a certain list of axioms which uniquely characterize the metric (see Section~\ref{subsec:lqg_metrics}). The limiting metric $D_h$ in Theorem~\ref{thm:uniqueness_of_subcritical_lqg_metrics} is the metric associated with \eqref{eq:lqg_metric_tensor} with parameter $\gamma \in (0,2)$, where $\xi$ and $\gamma$ are related by 
\begin{equation}\label{eq:gamma_xi}
Q(\xi) = \frac{2}{\gamma} + \frac{\gamma}{2}.  
\end{equation}

More generally, the limit in probability of $\alpha_{\varepsilon}^{-1} D_h^{\varepsilon}$ as $\varepsilon \to 0$ exists for all $\xi>0$ whenever $h$ is a whole-plane GFF plus a bounded and continuous function. However, the limit is not a continuous metric for $\xi > \xi_c$ but it is lower semicontinuous and takes infinite values. Roughly speaking, the reason for this is the existence of points $z \in \CC$ of ''thickness" greater than $Q \in (0,2)$, i.e., points for which
\begin{equation*}
    \limsup_{\varepsilon \to 0} \frac{h_{\varepsilon}(z)}{\log \varepsilon^{-1}} > Q,
\end{equation*}
where $h_{\varepsilon}(z)$ denotes the average of $h$ over the circle centered at $z$ with radius $\varepsilon$. These points are called \emph{singular} for the limiting metric.
\begin{definition}\label{def:lower_semi_cont}
We say that a sequence of lower semicontinuous functions $(f_n)$ on $\CC \times \CC$ converges to a function $f$ on $\CC \times \CC$ with respect to the lower semicontinuous topology if and only if the following hold.
\begin{enumerate}
    \item For every sequence of points $(z_n,w_n) \in \CC \times \CC$ converging to a point $(z,w) \in \CC \times \CC$, we have that
    \begin{equation*}
        f(z,w) \leq \liminf_{n \to \infty} f_n(z_n,w_n).
    \end{equation*}
    \item For every point $(z,w) \in \CC \times \CC$, there exists a sequence of points $(z_n,w_n) \in \CC \times \CC$ converging to $(z,w)$ for which
    \begin{equation*}
        f(z,w) = \lim_{n \to \infty} f_n(z_n,w_n).
    \end{equation*}
\end{enumerate}  
\end{definition}
It is easy to see that if $f_n \rightarrow f$ as $n \to \infty$ in the sense of Definition~\ref{def:lower_semi_cont}, then $f$ is lower semicontinuous. Moreover, it is easy to verify that the set of lower semicontinuous functions $\CC \times \CC \rightarrow \RR \cup \{\pm \infty\}$ endowed by the convergence in Definition~\ref{def:lower_semi_cont} is metrizable and the resulting metric space is separable and complete (see e.g., \cite[Section~1.2]{ding2023uniqueness}).

The following was shown in \cite{ding2023uniqueness}.
\begin{theorem}(\cite[Theorem~1.3]{ding2023uniqueness})\label{thm:uniqueness_supercritcal}
Let $h$ be a whole-plane GFF plus a bounded continuous function. For each $\xi>0$, the re-scaled LFPP metrics $\alpha_{\varepsilon}^{-1} D_h^{\varepsilon}$ converge in probability as $\varepsilon \to 0$ with respect to the topology of Definition~\ref{def:lower_semi_cont} to a limit $D_h^{\xi}$. The limit $D_h^{\xi}$ is a random metric on $\CC$, except that it is allowed to take on infinite values.
\end{theorem}
As in \cite{GM21}, it was shown in \cite{ding2023uniqueness} that the limiting metric in Theorem~\ref{thm:uniqueness_supercritcal} satisfies a list of axioms which uniquely characterize it modulo a multiplicative deterministic constant (see Section~\ref{subsec:lqg_metrics}).

As for the critical regime, it was shown in \cite{ding2024critical} that the critical LQG metric induces the Euclidean topology almost surely and it corresponds to LQG with $\gamma = 2$. However, its modulus of continuity with respect to the Euclidean metric behaves differently from the subcritical LQG metric. In particular, the former has a logarithmic modulus of continuity (see \cite[Theorem~1.7]{ding2024critical}) while the latter is locally H\"older continuous with respect to the Euclidean metric (see \cite[Theorem~1.7]{DFGPS20}). Roughly speaking, the reason for this difference is the following. First, we note that it was shown in \cite{hu2010thick} that it is almost surely the case that there are no points $z \in \CC$ such that
\begin{equation}\label{eq:thick_points}
\limsup_{\varepsilon \to 0} \frac{h_{\varepsilon}(z)}{\log \varepsilon^{-1}} = \alpha
\end{equation}
for some $|\alpha|>2$. It was also shown in \cite{hu2010thick} that is is almost surely the case that for each $\alpha \in [-2,2]$ there exist points $z \in \CC$ for which \eqref{eq:thick_points} holds. Therefore, it follows that $\xi_c$ is the critical threshold for which singular points exist.

\subsection{Main result}
\label{subsec:main_result}

Let us now describe the main result of the paper. For all $\xi>0$, we let $D_h^{\xi}$ denote the limiting metric in Theorem~\ref{thm:uniqueness_supercritcal} whenever $h$ is a whole-plane GFF plus a bounded continuous function. It is then natural to ask whether the family of metrics $(D_h^{\xi_n})_{n \in \NN}$ (appropriately re-normalized) is tight or it converges in law or in probability to some non-trivial limit whenever $(\xi_n)_{n \in \NN}$ is a sequence in $(0,\infty)$ converging to $\xi \in (0,\infty)$ as $n \to \infty$. The paper \cite{kavvadias2025continuity} gave an affirmative answer to the above question in the case that $(\xi_n)_{n \in \NN}$ is a sequence in $(0,\xi_c)$ and $\xi \in [0,\xi_c)$. Here, we interpret the LQG metric with parameter $\xi = 0$ (equivalently, $\gamma=0$) as the Euclidean metric. In fact, it was shown in \cite[Theorems~1.5 and ~1.7]{kavvadias2025continuity} that the law of the $\xi$-LQG metric (appropriately re-normalized) is continuous in $\xi \in [0,\xi_c)$ with respect to the topology of local uniform convergence in $\CC \times \CC$ (see also \cite[Theorems~1.6 and ~1.9]{kavvadias2025continuity} for results regarding tightness for general $\xi$ with respect to the lower semicontinuous topology of functions).

Note that the paper \cite{kavvadias2025continuity} used the H\"older continuity estimates obtained in \cite[Section~3]{DFGPS20} combined with an appropriate re-normalization of the metric to obtain that the aforementioned estimates are uniform in $\gamma \in (0,2)$ as long as $\gamma$ is bounded away from $2$. However, the argument used in \cite{kavvadias2025continuity} does not work when $\gamma \to 2$ since the estimates obtained in \cite[Section~3]{DFGPS20} do not work in the $\gamma = 2$ case (see \cite[Theorem~1.7 and Proposition~1.8]{ding2024critical}). The goal of this paper is to consider the limit in probability as $\gamma \to 2$ of the $\gamma$-LQG metric (appropriately re-normalized) with respect to the local uniform topology of functions on $\CC \times \CC$. In particular, we will prove the following.
\begin{theorem}\label{thm:main_result}
Let $h$ be a whole-plane GFF normalized such that $h_1(0) = 0$. For all $\xi>0$, we set
\begin{equation}\label{eq:renormalization}
\beta(\xi):=\inf\left\{l>0 : \mathbb{P}[D_h^{\xi}(0,1) \leq l] > \frac{1}{2} \right\}  
\end{equation}
 and $\widehat{D}_h^{\xi}:=\beta(\xi)^{-1} D_h^{\xi}$, and note that Lemma~\ref{lem:normalization_well_defined} implies that
 \begin{equation*}
     \mathbb{P}[D_h^{\xi}(0,1) \leq \beta(\xi)] = \frac{1}{2}.
 \end{equation*}
Then, the following is true. Let $(\xi_n)_{n \in \NN}$ be a sequence in $(0,\xi_c)$ such that $\xi_n \to \xi_c$ as $n \to \infty$. Then, we have that
 \begin{equation*}
     \widehat{D}_h^{\xi_n} \to \widehat{D}_h^{\xi_c} \quad \text{in probability}\quad \text{as} \quad n \to \infty
 \end{equation*}
 with respect to the local uniform topology of functions on $\CC \times \CC$.
\end{theorem}

Let us now briefly explain why \cite[Theorems~1.5 and ~1.7]{kavvadias2025continuity} are non-trivial to prove. First, we note that it is explained in \cite[Section~1.2]{kavvadias2025continuity} that is is easy to see that the law of the volume measure associated with \eqref{eq:lqg_metric_tensor} is continuous in $\gamma \in (0,2)$ (with respect to the weak convergence of measures) using Kahane's convexity inequality and Kolmogorov's continuity criterion. However, a similar reasoning does not apply for the LFPP metrics $D_h^{\varepsilon}$ since they are not defined as a direct regularization of the field $h$ (as it is the case for the $\gamma$-LQG volume form) due to the presence of the infimum over all paths $P : [0,1] \rightarrow \CC$ which are piecewise differentiable in \eqref{eq:lfpp_metric}. The same difficulty is present for all $\xi>0$. Furthermore, the LQG area measure has been constructed in the critical case $\gamma = 2$ (see \cite{duplantier2014critical,duplantier2014renormalization}), but the construction is more difficult than for $\gamma \in (0,2)$. Also, if we re-normalize the subcritical $\gamma$-LQG area measure by $1 / (2-\gamma)$, then we obtain the critical LQG are measure in the limit as $\gamma \to 2$ (see \cite[Theorem~1.1]{aru2019critical}). In particular, the volume form associated with \eqref{eq:lqg_metric_tensor} becomes degenerate as $\gamma \to 2$. This suggests that it is non-trivial to show that the limit as $\gamma \to 2$ of the $\gamma$-LQG metric is non-trivial.

Finally, we state two consequences of Theorem~\ref{thm:main_result}.
\begin{corollary}\label{cor:main_result}
Suppose that we have the same setup as in Theorem~\ref{thm:main_result}. Let $(f_n)_{n \in \NN},f$ be real-valued bounded continuous functions on $\CC$ such that $f_n \to f$ as $n \to \infty$ locally uniformly and $\sup_{n \in \NN} ||f_n||_{\infty} < \infty$. Then, we have that
\begin{equation*}
    \widehat{D}_{h+f_n}^{\xi_n} \to \widehat{D}_{h+f}^{\xi_c} \quad \text{in probability} \quad \text{as} \quad n \to \infty
\end{equation*}
with respect to the local uniform topology of functions on $\CC \times \CC$.    
\end{corollary}

\begin{corollary}\label{cor:all_gamma}
Suppose that we have the same setup as in Theorem~\ref{thm:main_result}. Fix $\xi \in [0,\xi_c]$ and let $(\xi_n)_{n \in \NN}$ be a sequence in $(0,\xi_c)$ such that $\xi_n \to \xi$ as $n \to \infty$. Then, we have that
\begin{equation*}
    \widehat{D}_h^{\xi_n} \to \widehat{D}_h^{\xi} \quad \text{in probability} \quad \text{as} \quad n \to \infty
\end{equation*}
with respect to the local uniform topology of functions on $\CC \times \CC$, where we denote by $\widehat{D}_h^0$ the Euclidean metric on $\CC$.    
\end{corollary}

\begin{proof}[Proof of Corollary~\ref{cor:all_gamma} assuming Theorem~\ref{thm:main_result}.]
It follows immediately from combining Theorem~\ref{thm:main_result} with \cite[Theorems~1.5 and 1.7]{kavvadias2025continuity}.    
\end{proof}

\subsection{Basic notation}
For all $z \in \CC , r>0$, we denote by $B_r(z)$ the Euclidean ball centered at $z$ with radius $r$.

For all $z \in \CC$ and all $0<r_1<r_2$, we set $\mathbb{A}_{r_1,r_2}(z):=B_{r_2}(z) \setminus \overline{B_{r_1}(z)}$. Moreover, for any set $A \subseteq \CC$ and all $r>0$, we denote by $B_r(A)$ the Euclidean neighborhood of $A$ with radius $r$.

For any two integers $a,b \in \NN$ with $a<b$, we set $[a,b]_{\ZZ}:=[a,b] \cap \NN$ and we denote by $|\cdot|$ the Euclidean metric on $\CC$.

For a metric $D: \CC \times \CC \rightarrow \RR \cup \{\pm \infty\}$ and an annular region $\mathbb{A} \subseteq \CC$, we define $D(\text{across}\,\, \mathbb{A})$ to be the $D$-distance between the inner and outer boundaries of $\mathbb{A}$. We also define $D(\text{around} \,\,  \mathbb{A})$ to be the infimum of the $D$-lengths of the paths in $\mathbb{A}$ which disconnect the inner and outer boundaries of $\mathbb{A}$. Moreover, for sets $A,B \subseteq \CC$, we define the distance between $A$ and $B$ with respect to $D$ by
\begin{equation*}
    D(A,B):=\inf\{D(x,y) : x \in A, y \in B\}.
\end{equation*}

For a distribution (generalized function) $h$ on $\CC$ and all $z \in \CC, r>0$, we will denote by $h_r(z)$ the average of $h$ on $\partial B_r(z)$.

We abbreviate \emph{almost surely} by a.s. and we denote by $||f||_{\infty}$ the supremum norm of a function $f : \mathbb{C} \rightarrow \mathbb{R}$.

If $f: (0,\infty) \rightarrow \RR$ and $g : (0,\infty) \rightarrow (0,\infty)$, we say that $f(A) = O_A(g(A))$ as $A \to \infty$ if $\frac{f(A)}{g(A)}$ remains bounded as $A \to \infty$.

Finally, for an open set $D \subseteq \CC$, we will denote by $C_0^{\infty}(D)$ the space of smooth and compactly supported functions $f : \CC \rightarrow \RR$ whose support is a subset of $D$.

\subsection{Outline}
\label{subsec:outline}

Now we explain the main ideas in the proof of Theorem~\ref{thm:main_result}. First, we are going to work with a different re-normalization than the one given in the statement of Theorem~\ref{thm:main_result}. In particular, we let $\alpha : (0,\infty) \rightarrow (0,\infty)$ denote the function introduced in \cite[Section~7]{kavvadias2025continuity} (we are going to define $\alpha$ explicitly in Section~\ref{subsec:lqg_metrics_renormalization}) and set $\widetilde{D}_h^{\xi}:=\alpha(\xi)^{-1} D_h^{\xi}$ for all $\xi>0$, where $h$ is a whole-plane GFF normalized such that $h_1(0) = 0$ and $D_h^{\xi}$ denotes the limiting metric in Theorem~\ref{thm:uniqueness_supercritcal}. Then, we are going to use the following result proved in \cite{kavvadias2025continuity}.
\begin{theorem}(\cite[Theorem~8.19]{kavvadias2025continuity})\label{thm:tightness_lqg_metrics}
Suppose that we have the setup described above and let $(\xi_n)_{n \in \NN}$ be a sequence in $(0,\infty)$ such that $\xi_n \to \xi$ as $n \to \infty$ for some $\xi \in (0,\infty)$. Then, the sequence of metrics $(\widetilde{D}_h^{\xi_n})_{n \in \NN}$ is tight with respect to the topology induced by the convergence in Definition~\ref{def:lower_semi_cont}, and let $(\xi_{k_n})_{n \in \NN}$ be a subsequence such that $(\widetilde{D}_h^{\xi_{k_n}})_{n \in \NN}$ converges in law to some random function $\widetilde{D} : \CC \times \CC \rightarrow \RR \cup \{\pm \infty\}$. Suppose that $\widetilde{D}$ satisfies the triangle inequality a.s. Then, there exists a $\xi$-LQG metric $h \rightarrow \widetilde{D}_h$ such that
\begin{equation*}
    \widetilde{D}_h^{\xi_{k_n}} \rightarrow \widetilde{D}_h \quad \text{in probability} \quad \text{as} \quad n \to \infty.
\end{equation*}  
\end{theorem}

Recall that limits of sequences of continuous functions with respect to the local uniform topology satisfy the triangle inequality and that local uniform convergence implies convergence in the sense of Definition~\ref{def:lower_semi_cont}. Thus, applying Theorem~\ref{thm:tightness_lqg_metrics} with $\xi = \xi_c$ and combining with the uniqueness (up to a multiplicative constant) of LQG metrics (see Section~\ref{subsec:lqg_metrics}), it suffices to prove that the sequence of metrics $(\widetilde{D}_h^{\xi_n})_{n \in \NN}$ is tight with respect to the local uniform topology of functions on $\CC \times \CC$ whenever $(\xi_n)_{n \in \NN}$ is a sequence in $(0,\xi_c)$ such that $\xi_n \to \xi_c$ as $n \to \infty$.

To prove the latter claim, we will prove the following quantitative estimate.
\begin{proposition}\label{prop:quantitative_estimate}
Fix $\theta \in (0,\xi_c / 8), \varepsilon \in (0,1)$, and let $U \subseteq \CC$ be a bounded open set. Then, there exists a deterministic constant $C \in (0,\infty)$ (depending only on $\varepsilon,\theta,(\xi_n)_{n \in \NN}$, and $U$) such that the following holds for all $n \in \NN$. With probability at least $1-\varepsilon$, we have that
\begin{equation*}
   \widetilde{D}_h^{\xi_n}(z,w) \leq C \left(\max\left\{1,\log\left(\frac{1}{|z-w|}\right)\right\}\right)^{-\theta} \quad \text{for all} \quad z,w \in U.
\end{equation*}    
\end{proposition}
Then, the tightness of $(\widetilde{D}_h^{\xi_n})_{n \in \NN}$ with respect to the local uniform topology on $\CC \times \CC$ will follow from combining Proposition~\ref{prop:quantitative_estimate} with the Skorokhod representation theorem.

To prove Proposition~\ref{prop:quantitative_estimate}, we will use an argument which is similar to the argument used for the proof of \cite[Theorem~1.7]{ding2024critical}. In particular, Proposition~\ref{prop:quantitative_estimate} will follow from two key estimates.

\textbf{First key estimate.} The first key estimate for the proof of Proposition~\ref{prop:quantitative_estimate} is the following (see Proposition~\ref{prop:main_lqg_estimate}). Fix $N>1$ and a Euclidean annulus $\mathbb{A}$. Then, we can choose the re-normalization $\widetilde{D}_h^{\xi}$ such that the following holds. There exists a constant $C \in (0,\infty)$ depending only on $\mathbb{A}$ such that we have for all $z \in \CC, S>1, r>0$, and $\xi>0$, that
\begin{equation}\label{eq:main_quantitative_estimate}
\mathbb{P}\left[\max\left\{\widetilde{D}_h^{\xi}(\text{around}\,\, r \mathbb{A} + z) , \widetilde{D}_h^{\xi}(\text{across}\,\, r \mathbb{A} + z)\right\} > S r^{\xi Q(\xi)} e^{\xi h_r(z)} \right ] \leq C S^{-N}.
\end{equation}

\textbf{Second key estimate.} The second key estimate is a tail bound for the maximum of the circle average process which is obtained in \cite[Proposition~2.4]{ding2024critical}. It states that for any bounded open set $U \subseteq \CC$ and any $\alpha \in (0,1/4)$, a.s. there exists (random) $C \in (0,\infty)$ such that
\begin{equation}\label{eq:circle_average_estimate}
|h_{e^{-n}}(z)| \leq 2 n - \alpha \log n + C \quad \text{for all} \quad n \in \NN, z \in (e^{-n-100} \ZZ^2) \cap U.    
\end{equation}

\textbf{Conclusion of the proof.} Next, we explain how we will use the two key estimates to obtain Proposition~\ref{prop:quantitative_estimate}. 

Fix $\varepsilon \in (0,1)$ small, $A \in (2,\infty)$ large, and let $U \subseteq \CC$ be a bounded open set. Then, using the independence properties of the  GFF (see \cite[Lemma~3.4]{ding2024critical}), we obtain a version of \eqref{eq:main_quantitative_estimate} when we \emph{condition} on the circle average $h_r(z)$ (see Lemma~\ref{lem:conditional_circle_average}). Then, by taking a union bound over all $z \in (e^{-m-100} \ZZ^2) \cap U$ and combining with standard estimates on circle averages of the GFF, we obtain that the following holds (see Lemmas~\ref{lem:main_estimate} and ~\ref{lem:uniform_circle_ave}). There exists constant $C \in (0,\infty)$ depending only on $\varepsilon,A,U$, and the sequence $(\xi_n)_{n \in \NN}$ such that the following holds for all $n \in \NN$. With probability at least $1-\varepsilon$, we have for all $m \in \NN, z \in U$, that
\begin{align}\label{eq:bounds_on_distances}
&\max\left\{\widetilde{D}_h^{\xi_n}(\text{around}\,\, \mathbb{A}_{e^{-m-51/100},e^{-m-1/2}}(z)) , \widetilde{D}_h^{\xi_n}(\text{across}\,\, \mathbb{A}_{e^{-m-100},e^{-m}}(z))\right\}\\\nonumber
&\leq C \exp\left(\xi_n h_{e^{-m}}(z) - \xi_n Q(\xi_n) m + \left(\frac{2m - h_{e^{-m}}(z)}{A}\right)\right).   
\end{align}

Suppose that we are working on the event that \eqref{eq:bounds_on_distances} occurs. Fix $k,m \in \NN$ with $k<m$. By stringing together paths in the annuli $\mathbb{A}_{e^{-\ell-100},e^{-\ell}}(z)$ and $\mathbb{A}_{e^{-\ell-51/100},e^{-\ell-1/2}}(z)$ for $\ell = 0,\cdots,m$, we obtain from \eqref{eq:bounds_on_distances} (see Lemma~\ref{lem:stringing_paths}) that for all $z \in U$,
\begin{align}\label{eq:sum_of_distances}
\widetilde{D}_h^{\xi_n}(\text{across}\,\, \mathbb{A}_{e^{-m},e^{-k}}(z)) \leq C \sum_{\ell = k}^m \exp\left(\xi_n h_{e^{-\ell}}(z) - \xi_n Q(\xi_n) \ell + (2+2\xi_n) \left(\frac{2\ell - h_{e^{-\ell}}(z)}{A}\right)\right).   
\end{align}
Next, using continuity estimates for the circle average process, we can replace the sum by an integral on the right side of \eqref{eq:sum_of_distances}. Moreover, combining with the lower semicontinuity of the metric and taking $m \to \infty$ on both sides of \eqref{eq:sum_of_distances} gives for all $z \in U$, that
\begin{align}\label{eq:integral_bound}
\widetilde{D}_h^{\xi_n}(z,\partial B_{e^{-k}}(z)) \leq C \int_k^{\infty} \exp\left(\xi_n h_{e^{-t}}(z) - \xi_n Q(\xi_n) t + (2+2\xi_n) \left(\frac{2t - h_{e^{-t}}(z)}{A}\right) \right) dt.  
\end{align}

Now, we would like to conclude the proof of Proposition~\ref{prop:quantitative_estimate} by combining \eqref{eq:circle_average_estimate} with \eqref{eq:integral_bound}. Then, by possibly taking $C$ to be larger (in a way that depends only on $\varepsilon,A,U$, and the sequence $(\xi_n)_{n \in \NN}$), we would obtain that the term on the right side of \eqref{eq:integral_bound} is bounded from above by 
\begin{equation}\label{eq:refined_integral_bound}
C \int_k^{\infty} \exp\left(-\alpha \xi_n \log t + (2-Q(\xi_n)) \xi_n t + \left(\frac{2+2\xi_n}{A}\right) (2t - h_{e^{-t}}(z)) \right) dt.  
\end{equation}
Note that $2t - h_{e^{-t}}(z) > 0$ for all $z \in U$ and all $t>0$ sufficiently large by \eqref{eq:circle_average_estimate}. Thus, since $\alpha \xi_c <1$ and $Q(\xi_n) \downarrow 2$ as $n \to \infty$, we obtain that the integral in \eqref{eq:refined_integral_bound} converges to $\infty$ as $n \to \infty$, for all $k \in \NN$ fixed. This implies that we need more refined control on the maximum of the circle average process than what we get from \eqref{eq:circle_average_estimate}.

To deal with the above issue, we fix $\theta \in (0,\xi_c / 8)$ and $\beta > (1+\theta) / \xi_c$. We will argue as in \cite[Section~3]{ding2024critical}. In particular, we will show that for every point $z \in U$, the set of times $t$ for which $h_{e^{-t}}(z) \geq 2 t -\beta \log t$ is small. The main tool to prove this is \cite[Lemma~3.12]{ding2024critical} which roughly states that a Brownian bridge on $[0,T]$ cannot spend too much time above $2t - \beta \log t$ during the time interval $[T / 2 , T]$ if it is constrained to stay below $2t - \alpha \log t$ whenever $0<\alpha < \beta$. Therefore, combining with the fact that $t \rightarrow h_{e^{-t}}(z) - h_1(z)$ is a standard linear Brownian motion, we obtain that we can control the amount of time that $h_{e^{-t}}(z)$ spends above $2t - \beta \log t$. This will give us a sharper upper bound on $\widetilde{D}_h^{\xi_n}(z,\partial B_{e^{-k}}(z))$ than the one obtained in \eqref{eq:refined_integral_bound} (see Lemmas~\ref{lem:distance_circle_large} and ~\ref{lem:large_at_some_t}).

\textbf{Outline of Section~\ref{sec:proof_of_main_result}.} Section~\ref{sec:proof_of_main_result} is the core part of our proof, where we will make the arguments described above rigorous.

In Section~\ref{subsec:estimates_across_around_annuli}, we will prove the estimate \eqref{eq:bounds_on_distances}.

In Section~\ref{subsec:stringing_paths}, we will explain how we will use \eqref{eq:bounds_on_distances} to string together paths and obtain a version of \eqref{eq:sum_of_distances}.

In Section~\ref{subsec:distances_between_scales}, we will give an upper bound on the time that $h_{e^{-t}}(z)$ spends above $2t - \beta \log t$. This will give us an upper bound on the sum on the right side of \eqref{eq:integral_bound}.

In Section~\ref{subsec:proofs_of_main_results}, we will conclude the proofs of Theorem~\ref{thm:main_result} and Corollary~\ref{cor:main_result}.

\textbf{Acknowledgements.} The author was supported by the Simons Collaboration Grant
\emph{Probabilistic Paths to Quantum Field Theory}.

\section{Preliminaries}
\label{sec:preliminaries}

\subsection{Gaussian free field}
\label{subsec:gff}

Throughout this paper, we shall work with a whole-plane GFF $h$ which is defined as the centered Gaussian process $h$ with covariances given by
\begin{equation}\label{eqn:gff_covariance}
    \Cov(h(u) , h(v)) = G(u,v) = \log \,\,\frac{\max\{|u|,1\}\max\{|v|,1\}}{|u-v|}
\end{equation}
for all $u,v \in \CC$. Since the covariance kernel $G(u,v)$ explodes along the diagonal, $h$ cannot be well-defined pointwise almost surely. However, it is well-defined as a distribution (generalized function) in the sense that almost surely, for any bump function $\phi$, the average $(h,\phi) = \int h(u) \phi(u) du$ is well-defined. Moreover, it is shown in \cite[Proposition~3.1]{DS09} that if $h_r(z)$ denotes the average of $h$ on $\partial B_r(z)$, there almost surely exists a version of $h$ such that the map $(z,r) \mapsto h_r(z)$ is almost surely continuous. Also, the renormalization in \eqref{eqn:gff_covariance} is chosen so that $h_1(0) = 0$. Furthermore, the law of $h$ is scale and translation invariant in the sense that for every fixed $z \in \CC, r>0$, the laws of the fields $h(\cdot),h(\cdot+z) - h_1(z)$ and $h(r \cdot) - h_r(0)$ are the same.

Let us define explicitly the topology on the set of distributions on $\CC$ that we are going to consider. For a simply connected domain $D \subseteq \CC$ with harmonically non-trivial boundary, we let $H_0(D)$ denote the Hilbert space closure of $C_0^{\infty}(\CC)$ with respect to the Dirichlet inner product 
\begin{equation*}
    (f,g)_{\nabla}:=\frac{1}{2\pi} \int_D \nabla f(z) \cdot \nabla g(z) dz.
\end{equation*}
Moreover, we let $H^{-1}(D)$ denote the dual space of $H_0(D)$. Let also $H_0(\CC)$ denote the Hilbert space closure with respect to $(\cdot,\cdot)_{\nabla}$ of the set of $f \in C_0^{\infty}(\CC)$ such that $\int_{\CC} f(z) dz = 0$.

The whole-plane GFF is a.s. an element of $H_{\text{loc}}^{-1}(\CC)$, where we denote by $H_{\text{loc}}^{-1}(\CC)$ the set of all generalized functions $\widetilde{h}$ on $\CC$ such that $\widetilde{h}|_{B_R(0)} \in H^{-1}(B_R(0))$ for all $R>0$. We can define a topology on $H_{\text{loc}}^{-1}(\CC)$ by requiring that a sequence $(\widetilde{h}_n)_{n \in \NN}$ in $H_{\text{loc}}^{-1}(\CC)$ converges to some $\widetilde{h} \in H_{\text{loc}}^{-1}(\CC)$ as $n \to \infty$ if and only if 
\begin{equation*}
\widetilde{h}_n|_{B_R(0)} \to \widetilde{h}|_{B_R(0)} \quad \text{as} \quad n \to \infty \quad \text{in} \quad H^{-1}(B_R(0)) \quad \text{for all} \quad R>0.    
\end{equation*}

It is easy to see (see e.g., \cite[Sections~1.6 and 1.7]{powell2025gaussian}) that $H_{\text{loc}}^{-1}(\CC)$ becomes a separable and complete metric space with the above topology.

\subsection{Properties of GFF circle averages}
\label{subsec:gff_averages}

Next, we state two useful results about the circle average process of a whole-plane GFF $h$ normalized so that $h_1(0) = 0$. The first one (Proposition~\ref{prop:circle_average_bound}) is an estimate for the maximum of the GFF and it is one of the key estimates for the proof of Proposition~\ref{prop:quantitative_estimate} as explained in Section~\ref{subsec:outline}. The second one (Proposition~\ref{prop:conditional_independence}) will allow us to obtain a version of \eqref{eq:main_quantitative_estimate} when we condition on the circle average $h_r(z)$.

\begin{proposition}(\cite[Proposition~2.4]{ding2024critical}) \label{prop:circle_average_bound}
Let $U \subseteq \CC$ be a bounded open set, let $\alpha \in (0,1/4)$, and let $k \in \NN$. Almost surely, there is a random $C>0$ such that
\begin{equation*}
    |h_{e^{-n}}(z)| \leq 2n - \alpha \log n + C,\quad \text{for all} \quad n \in \NN, z \in (e^{-n-k} \ZZ^2) \cap U.
\end{equation*}
\end{proposition}

\begin{proposition}(\cite[Lemma~3.4]{ding2024critical}) \label{prop:conditional_independence}
For each $t \in \RR$, the process $\{h_{e^{-s}}(0) - h_{e^{-t}}(0) : s \leq t \}$ is independent from $(h-h_{e^{-t}}(0))|_{B_{e^{-t}}(0)}$.  
\end{proposition}

\subsection{LQG metrics}
\label{subsec:lqg_metrics}

As discussed in Sections~\ref{subsec:overview}-\ref{subsec:main_result}, the GFF $h$ comes associated with a metric on $\CC$ and the latter is uniquely determined (modulo a deterministic multiplicative constant) via a list of axioms. Before stating the aforementioned axioms, we will give some preliminary definitions.

\begin{definition}
    Let $(X,d)$ be a metric space with $d$ allowed to take on infinite values.
    \begin{itemize}
        \item A \textbf{path} in $(X,d)$ is a continuous function $P : [a,b] \rightarrow X$ for some interval $[a,b]$.
        \item For a curve $P : [a,b] \rightarrow X$, the $d$-\textbf{length} of $P$ is defined by 
        \begin{equation*}
            \text{len}(P ; d):=\sup_T \sum_{i=1}^{|T|}d(P(t_i),P(t_{i-1}))
        \end{equation*}
        where the supremum is over all partitions $T : a=t_0 <t_1<\cdots<t_{|T|} = b$ of $[a,b]$. Note that the $d$-length of a curve can be infinite.
        \item We say that $(X,d)$ is a \textbf{length space} if for all $x,y \in X, \varepsilon>0$, there exists a path in $X$ with $d$-length at most $d(x,y)+\varepsilon$ from $x$ to $y$. If $d(x,y) < \infty$, a curve in $X$ from $x$ to $y$ with $d$-length exactly $d(x,y)$ is called a \textbf{geodesic}. We say that $(X,d)$ is a \textbf{geodesic space} if for all $x,y \in X$, we have that $d(x,y) < \infty$ and there exists a geodesic in $X$ from $x$ to $y$.
        \item For $Y \subseteq X$, the \textbf{internal metric} of $d$ on $Y$ is defined by 
        \begin{equation*}
            d(x,y ; Y):=\inf_{P \subseteq Y} \text{len}(P ; d), \quad \text{for all} \quad x,y \in Y
        \end{equation*}
        where the infimum is over all paths $P$ in $Y$ from $x$ to $y$. We note that $d(\cdot,\cdot;Y)$ is a metric on $Y$.
        \item If $X \subseteq \CC$, we say that $d$ is a \textbf{lower semicontinuous metric} if the function $(x,y) \rightarrow d(x,y)$ is lower semicontinuous with respect to the Euclidean topology. We equip the set of lower semicontinuous metrics on $X$ with the topology of lower semicontinuous functions on $X \times X$ induced by the convergence in Definition~\ref{def:lower_semi_cont}.
    \end{itemize}
\end{definition}

We are now ready to state the axioms of a (strong) LQG metric with parameter $\xi>0$.

\begin{definition}\label{def:strong_lqg_metric}
    Let $\mathcal{D}'(\CC)$ denote the space of distributions (generalized functions) on $\CC$, equipped with the standard weak topology. Fix $\xi>0$. A \textbf{(strong) LQG metric with parameter $\xi$} is a measurable function $h \rightarrow D_h$ from $\mathcal{D}'(\CC)$ to the space of lower semicontinuous metrics on $\CC$ such that the following hold whenever $h$ is a whole-plane GFF plus a continuous function on $\CC$.
    \begin{enumerate}[(I)]
        \item \label{it:length_space} 
        \textbf{Length space.} A.s., $(\CC,D_h)$ is a length space.
        \item \label{it:locality}
        \textbf{Locality.} Let $U \subseteq \CC$ be a deterministic open set. The $D_h$-internal metric $D_h(\cdot,\cdot ; U)$ is a.s. determined by $h|_U$.
        \item \label{it:weyl_scaling}
        \textbf{Weyl scaling.} For each continuous function $f : \CC \rightarrow \RR$, we define
        \begin{equation*}
            (e^{\xi f} \cdot D_h)(z,w):=\inf_{P : z \rightarrow w} \left\{\int_0^{\text{len}(P;D_h)} e^{\xi f(P(t))} dt \right\} \quad \text{for all} \quad z,w \in \CC,
        \end{equation*}
        where the infimum is over all $D_h$-rectifiable paths from $z$ to $w$ parameterized by $D_h$-length. Then, we have that $e^{\xi f} \cdot D_h = D_{h+f}$ for all continuous functions $f : \CC \rightarrow \RR$ a.s.
        \item \label{it:scale_translation_invariance}
        \textbf{Scale and translation invariance.} Let $Q$ be as in \eqref{eq:lfpp_normalizing_factor}. For each fixed and deterministic $r>0$ and $z \in \CC$, a.s.
        \begin{equation*}
            D_h(ru+z,rv+z) = D_{h(r \cdot +z)+Q \log r}(u,v) \quad \text{for all} \quad u,v \in \CC.
        \end{equation*}
        \item \label{it:finiteness}
        \textbf{Finiteness.} Let $U \subseteq \CC$ be a deterministic, open, connected set and let $K_1, K_2 \subseteq U$ be disjoint, deterministic, compact, connected sets which are not singletons. Almost surely, $D_h(K_1,K_2 ; U) < \infty$.
    \end{enumerate}
\end{definition}

The following theorem states that the limiting metric in Theorem~\ref{thm:uniqueness_supercritcal} is a (strong) LQG metric with parameter $\xi$ that is uniquely determined via the axioms in Definition~\ref{def:strong_lqg_metric}.

\begin{theorem}(\cite[Theorem~1.8]{ding2023uniqueness})
\label{thm:uniqueness_lqg_metrics}
Fix $\xi>0$. Then there is an LQG metric $D$ with parameter $\xi$ such that the limiting metric of Theorem~\ref{thm:uniqueness_supercritcal} is a.s. equal to $D_h$ whenever $h$ is a whole-plane GFF plus a bounded continuous function. Furthermore, this LQG metric is unique in the following sense. If $D$ and $\widetilde{D}$ are two LQG metrics with parameter $\xi$, then there exists a deterministic constant $C \in (0,\infty)$ such that if $h$ is a whole-plane GFF plus a continuous function, then a.s. $D_h = C \widetilde{D}_h$.  
\end{theorem}

\subsection{Re-normalized LQG metrics}
\label{subsec:lqg_metrics_renormalization}

As explained in Section~\ref{subsec:outline}, to prove Theorem~\ref{thm:main_result}, we will first re-normalize the metrics $D_h^{\xi}$ in Theorem~\ref{thm:uniqueness_supercritcal} in a different way. This will allow us to apply the results in \cite[Sections~7 and ~8]{kavvadias2025continuity}. Before defining explicitly the above re-normalization, we state two lemmas that we are going to use. For the rest of the subsection, we will assume that $h$ is a whole-plane GFF normalized so that $h_1(0) = 0$.

\begin{lemma}(\cite[Lemma~3.3]{kavvadias2025continuity})
\label{lem:high_prob}
For each $M>0$, there exists $\mathfrak{p}_0 \in (0,1)$ depending only on $M$ such that the following holds for all $\mathfrak{p} \in (\mathfrak{p}_0,1)$. Let $\{E_r(z)\}_{z \in \CC,r>0}$ be a collection of events such that for all $z \in \CC,r>0$, we have that $\mathbb{P}[E_r(z)] = \mathbb{P}[E_1(0)] \geq \mathfrak{p}$ and $E_r(z)$ is measurable with respect to the $\sigma$-algebra generated by 
\begin{equation*}
    (h-h_{3r}(z))|_{\mathbb{A}_{r,2r}(z)}.
\end{equation*}
Then, for all $\mathfrak{r}>0$, it holds with probability at least $1 - \varepsilon^M$ as $\varepsilon \to 0$, at a rate which is universal, that the following is true. For each $z \in B_{\mathfrak{r} \varepsilon^{-M}}(0)$, there exist 
\begin{equation*}
    r \in [\varepsilon^2 \mathfrak{r},\varepsilon \mathfrak{r}] \cap \{2^{-k} \mathfrak{r}\}_{k \in \NN}, w \in B_{\mathfrak{r} \varepsilon^{-M}}(0) \cap \left(\frac{\varepsilon^2 \mathfrak{r}}{4} \ZZ^2\right)
\end{equation*}
such that $E_r(w)$ occurs and $z \in B_{\mathfrak{r} \varepsilon^2 / 2}(w)$.
\end{lemma}

\begin{remark}
We note that \cite[Lemma~3.3]{kavvadias2025continuity} is stated in the case that $E_r(z)$ is the event that
\begin{equation*}
    \widetilde{D}_h^{\xi}(\mathrm{around}\,\, \mathbb{A}_{r,2r}(z)) \leq r^{\xi Q(\xi)} e^{\xi h_r(z)}
\end{equation*}
for fixed $z \in \mathbb{C},  r>0$, where $\widetilde{D}_h^{\xi}$ denotes an appropriate re-normalization of $D_h^{\xi}$. However, the proof of \cite[Lemma~3.3]{kavvadias2025continuity} only requires that $E_r(z)$ is measurable with respect to the $\sigma$-algebra $\sigma((h-h_{3r}(z))|_{\mathbb{A}_{r,2r}(z)})$ and that $\mathbb{P}[E_r(z)] = \mathbb{P}[E_1(0)] = \mathfrak{p}$ for each $z \in \mathbb{C},r>0$, and some constant $\mathfrak{p} \in (0,1)$ depending only on $M$.
\end{remark}

\begin{lemma}
\label{lem:normalization_well_defined}
Fix $\xi>0$. Then, for all $x>0$, we have that
\begin{equation*}
    \mathbb{P}[D_h^{\xi}(\text{around} \,\, \mathbb{A}_{1,2}(0)) = x] = \mathbb{P}[D_h^{\xi}(0,1) = x] = 0.
\end{equation*}   
Moreover, we have that
\begin{equation*}
    \mathbb{P}[D_h^{\xi}(0,1) \in (a,b)] > 0 \quad \text{for all} \quad 0<a<b.
\end{equation*}
\end{lemma}

\begin{proof}
The first claim of the lemma follows from \cite[Lemma~10.1]{kavvadias2025continuity}. For the rest of the proof, we will focus on proving the second claim of the lemma.

Note that $D_h^{\xi}(0,1) \in (0,\infty)$ a.s. First, we show that there exists $x \in (0,\infty)$ such that
\begin{equation}\label{eq:support_of_measure}
\mathbb{P}[D_h^{\xi}(0,1) \in (x-\varepsilon,x+\varepsilon) ] > 0 \quad \text{for all} \quad \varepsilon > 0.   
\end{equation}
Suppose that this is not the case. Then, for all $x>0$, there exists $\varepsilon_x \in (0,x)$ such that
\begin{equation*}
\mathbb{P}[D_h^{\xi}(0,1) \in (x-\varepsilon_x,x+\varepsilon_x)] = 0.
\end{equation*}
Note that the family of intervals $\{(x-\varepsilon_x,x+\varepsilon_x)\}_{x>0}$ is an open cover of $(0,\infty)$. Since $(0,\infty)$ is second countable when endowed with the Euclidean topology, we obtain that there exists a sequence of points $(x_n)_{n \geq 1}$ such that $(0,\infty) = \cup_{n \in \mathbb{N}} (x_n - \varepsilon_{x_n} , x_n + \varepsilon_{x_n})$. But then, we have that
\begin{equation*}
\mathbb{P}[D_h^{\xi}(0,1) \in (0,\infty)] \leq \sum_{n \in \mathbb{N}} \mathbb{P}[D_h^{\xi}(0,1) \in (x_n - \varepsilon_{x_n} , x_n + \varepsilon_{x_n})] = 0   
\end{equation*}
but that is a contradiction. It follows that there exists $x>0$ such that \eqref{eq:support_of_measure} holds.

Next, we fix $y,\varepsilon>0$ and let $c \in \mathbb{R}$ be such that $y = x e^{\xi c}$. Note that \cite[Lemma~1.16]{ding2023uniqueness} implies that it is a.s. the case that the points $0$ and $1$ can be joined by a $D_h^{\xi}$-geodesic $P$. Moreover, we have by \cite[Proposition~5.19]{ding2023tightness} that a.s. for all $n>0$,
\begin{equation*}
D_h^{\xi}(\partial B_n(0) , \partial B_N(0)) \to \infty \quad \text{as} \quad N \to \infty.   
\end{equation*}
Thus, combining with \eqref{eq:support_of_measure}, we obtain that we can choose $1<n<N$ such that with positive probability, we have that
\begin{align}\label{eq:all_conditions}
D_h^{\xi}(0,1) \in (x-\varepsilon e^{-\xi c} , x+\varepsilon e^{-\xi c}), P \subseteq B_n(0), \,\,\text{and} \,\,  D_h^{\xi}(\partial B_n(0),\partial B_N(0)) > e^{3 \xi |c|} D_h^{\xi}(0,1).   
\end{align}

Let $f: \mathbb{C} \rightarrow \mathbb{R}$ be a smooth and compactly supported function such that $f \equiv c$ on $B_N(0)$, $||f||_{\infty} \leq 2 |c|$ and $f\equiv 0$ on $\mathbb{C} \setminus B_{2N}(0)$. Then, it follows from combining \eqref{eq:all_conditions} with Axiom~\eqref{it:weyl_scaling} that 
\begin{equation*}
D_{h+f}^{\xi}(\partial B_n(0) , \partial Be_N(0)) \geq e^{-\xi ||f||_{\infty}} D_h^{\xi}(\partial B_n(0) , \partial B_N(0)) > e^{\xi c} D_h^{\xi}(0,1).   
\end{equation*}
Hence, it follows from Axiom~\eqref{it:weyl_scaling} that if the event in \eqref{eq:all_conditions} occurs, we have that
\begin{equation*}
D_{h+f}^{\xi}(0,1) = D_{h+f}^{\xi}(0,1 ; B_N(0)) = e^{\xi c} D_h^{\xi}(0,1),   
\end{equation*}
which implies that
\begin{equation}\label{eq:positive_prob}
\mathbb{P}[D_{h+f}^{\xi}(0,1 ; B_N(0)) \in (y-\varepsilon,y+\varepsilon), D_{h+f}^{\xi}(\partial B_n(0) , \partial B_N(0)) > D_{h+f}^{\xi}(0,1 ; B_N(0))] > 0.    
\end{equation}

Note that by Axiom~\eqref{it:locality}, the event in \eqref{eq:positive_prob} is a.s. determined by $(h+f)|_{B_N(0)}$. Note also that by \cite[Proposition~2.9]{IG4}, the laws of the fields $h|_{B_N(0)}$ and $(h+f)|_{B_N(0)}$ are mutually absolutely continuous and the corresponding Radon-Nikodym derivatives are positive a.s. It follows that the event in \eqref{eq:positive_prob} defined with $h$ in place of $h+f$ has positive probability which implies that
\begin{equation*}
\mathbb{P}[D_h^{\xi}(0,1) \in (y-\varepsilon , y + \varepsilon)] > 0.   
\end{equation*}
This completes the proof of the lemma.    
\end{proof}

Now, we are ready to define explicitly the re-normalization that we are going to use. Set $\zeta(\xi) = Q(\xi) / 4$ for all $\xi>0$. We also let $b : (0,\infty) \rightarrow (0,\infty)$ be a continuous function and let $M : (0,\infty) \rightarrow (0,\infty)$ be a continuous and non-decreasing function such that
\begin{equation}\label{eq:condition_for_normalization_1}
    b(\xi) < \frac{1}{2\sqrt{\xi}}
\end{equation}
and
\begin{equation}\label{eq:condition_for_nprmalization_2}
    \frac{4-\xi Q(\xi) + \xi 2 \sqrt{2} \sqrt{2 + M(\xi)}}{\sqrt{M(\xi)}} < \frac{1}{b(\xi)},\,\, \xi \zeta(\xi) b(\xi) \sqrt{M(\xi)} > 1,\,\,\text{and}\,\, b(\xi) \sqrt{M(\xi)} > 1
\end{equation}
for all $\xi>0$.

For all $\xi>0$, we let $\mathfrak{p}_0(\xi) \in (0,1)$ be the minimum number in $(0,1)$ such that the statement of Lemma~\ref{lem:high_prob} holds with $M = M(\xi)$. We then define $\alpha(\xi) \in (0,\infty)$ by
\begin{equation}\label{eq:auxiliary_normalization}
\alpha(\xi):=\inf\left\{l>0 : \mathbb{P}[D_h^{\xi}(\text{around}\,\, \mathbb{A}_{1,2}(0)) \leq l] > \mathfrak{p}_0(\xi) \right\}.
\end{equation}
Note that Lemma~\ref{lem:normalization_well_defined} implies that
\begin{equation*}
    \mathbb{P}[D_h^{\xi}(\text{around}\,\,\mathbb{A}_{1,2}(0)) \leq \alpha(\xi)] = \mathfrak{p}_0(\xi)
\end{equation*}
for all $\xi>0$. We then set
\begin{equation}\label{eq:auxiliary_metric_normalization}
\widetilde{D}_h^{\xi}:=\alpha(\xi)^{-1} D_h^{\xi}.   
\end{equation}

Next, we state a useful estimate from \cite{kavvadias2025continuity} (which is uniform in $\xi$) on the $\widetilde{D}_h^{\xi}$-distance between two fixed, compact, connected, disjoint sets which are not singletons. It will be the main input for the proof of \eqref{eq:main_quantitative_estimate}. Recall that the latter is one of the two key estimates for the proof of Proposition~\ref{prop:quantitative_estimate}.

\begin{proposition}(\cite[Proposition~7.1]{kavvadias2025continuity})
\label{prop:main_lqg_estimate}
Suppose that we have the same setup described above. Let $U \subseteq \CC$ be open and connected and let $K_1,K_2 \subseteq U$ be connected , disjoint compact sets which are not singletons. Then, for all $r>0$, it holds with probability at least $1-O_A(A^{-b(\xi) \sqrt{M(\xi)}})$ as $A \to \infty$, at a rate which is uniform in the choice of $r$ and $\xi$ and depends only on $K_1,K_2$, and $U$, that
\begin{equation}\label{eq:lqg_bound}
    \widetilde{D}_h^{\xi}(r K_1,rK_2 ; r U) \leq A r^{\xi Q(\xi)} e^{\xi h_r(0)}.
\end{equation}  
More precisely, there exist constants $c,A_0>0$ depending only on $K_1,K_2$, and $U$, such that for all $A \geq A_0$, we have that \eqref{eq:lqg_bound} holds with probability at least $1-C A^{-b(\xi) \sqrt{M(\xi)}}$.
\end{proposition}

\section{Proof of the main result}
\label{sec:proof_of_main_result}

In this section, we are going to prove Theorem~\ref{thm:main_result} and Corollary~\ref{cor:main_result}. For the rest of the section, we will assume that $h$ is a whole-plane GFF normalized such that $h_1(0) = 0$. Moreover, for all $\xi>0$, we will denote by $\widetilde{D}_h^{\xi}$ the re-normalization of the metric $D_h^{\xi}$ introduced in \eqref{eq:auxiliary_metric_normalization} for some fixed functions $b,M$ satisfying \eqref{eq:condition_for_normalization_1} and \eqref{eq:condition_for_nprmalization_2}. In Sections~\ref{subsec:estimates_across_around_annuli}-\ref{subsec:proofs_of_main_results}, we will require that the functions $b,M$ satisfy some further properties which we are going to make clear in the relevant lemmas and propositions.

Finally, most of the estimates in Sections~\ref{subsec:estimates_across_around_annuli}-\ref{subsec:proofs_of_main_results} will be obtained for the metrics $\widetilde{D}_h^{\xi}$ instead of the metrics $\widehat{D}_h^{\xi}$ appearing in the statement of Theorem~\ref{thm:main_result}. The reason for this is that this will allow us to use the crucial estimate obtained in Proposition~\ref{prop:main_lqg_estimate} which is stated in terms of the metrics $\widetilde{D}_h^{\xi}$.

\subsection{Estimate of distances across and around annuli in terms of circle averages}
\label{subsec:estimates_across_around_annuli}

As explained in Section~\ref{subsec:outline}, to get \eqref{eq:sum_of_distances}, we will string together paths in annuli of the form $\mathbb{A}_{e^{-t-100},e^{-t}}(z)$ that cross the annuli and paths in annuli of the form $\mathbb{A}_{e^{-t-51/100},e^{-t-1/2}}(z)$ that disconnect the inner and outer boundaries of the annuli. To do this, we will bound $\widetilde{D}_h^{\xi}(\text{around} \,\, \mathbb{A}_{e^{-t-51/100},e^{-t-1/2}}(z))$ and $\widetilde{D}_h^{\xi}(\text{across}\,\, \mathbb{A}_{e^{-t-100},e^{-t}}(z))$ in terms of the circle average process associated with $h$.

In order to state the above estimates more conveniently and help readability, we introduce the following definition.

\begin{definition}\label{def:main_quantities}
For $\xi>0, z \in \CC$ and $t\geq 0$, we let $M_{e^{-t}}(z)$ be the maximum of the following quantities:
\begin{enumerate}
    \item \label{it:distance_around}
    $\exp(-\xi h_{e^{-t}}(z) + \xi Q(\xi) t) \widetilde{D}_h^{\xi}(\text{around}\,\, \mathbb{A}_{e^{-t-51/100},e^{-t-1/2}}(z))$;
    \item \label{it:distance_acroos}
    $\exp(-\xi h_{e^{-t}}(z) + \xi Q(\xi) t) \widetilde{D}_h^{\xi}(\text{across} \,\, \mathbb{A}_{e^{-t-100},e^{-t}}(z))$;
    \item \label{it:circle_average_not_large}
    $\sup_{r \in [e^{-t-100},e^{-t}]}\sup_{w \in B_{e^{-t}-r}(z)} \exp(|h_r(w)-h_{e^{-t}}(z)|)$.
\end{enumerate}
\end{definition}

The main goal of this subsection is to prove the following estimate for $M_{e^{-t}}(z)$ in terms of the circle average process associated with $h$.

\begin{lemma}\label{lem:main_estimate}
Fix $\varepsilon \in (0,1), A \in (0,\infty)$ and suppose that $b(\xi) \sqrt{M(\xi)} > 20 A$, and let $U \subseteq \CC$ be a bounded open set. Then, there exists a constant $C \in (0,\infty)$ depending only on $\varepsilon,A$, and $U$ such that the following holds with probability at least $1-\varepsilon$. For all $n \in \NN$ and all $z \in (e^{-n-100} \ZZ^2) \cap U$, we have that
\begin{equation*}
    M_{e^{-n}}(z) \leq C \exp\left(\frac{2n - h_{e^{-n}}(z)}{A}\right).
\end{equation*}    
\end{lemma}

We note that the statement of Lemma~\ref{lem:main_estimate} is similar to that of \cite[Lemma~3.2]{ding2024critical} but with the term $|2n-h_{e^{-n}}(z)|^{1/2+\zeta}$ for some $\zeta>0$ replaced by the term $(2n-h_{e^{-n}}(z)) / A$. We choose to state Lemma~\ref{lem:main_estimate} in this way since it will make it easier to verify that the required estimates are uniform in $\xi>0$. However, the overall strategy for the proof of Lemma~\ref{lem:main_estimate} will be similar to that of \cite[Lemma~3.2]{ding2024critical}. 

Let us now give a brief outline for the strategy of the proof of Lemma~\ref{lem:main_estimate}. First, in Lemma~\ref{lem:conditional_circle_average}, we will bound the conditional probability that $M_{e^{-t}}(z) > S$ given $h_{e^{-t}}(z) - h_1(z)$ uniformly in $z\in \CC$ and $t \geq 0$, and the upper bound on the conditional probability will be given in terms of $\xi$ and $S$. The main input in the proof of Lemma~\ref{lem:conditional_circle_average} will be the estimate in Proposition~\ref{prop:main_lqg_estimate}.

Next, using Lemma~\ref{lem:conditional_circle_average}, we will prove a version of Lemma~\ref{lem:main_estimate} (Lemma~\ref{lem:precise_bound}) which holds for a fixed value of $t>0$ instead of all $n \in \NN$. Lemma~\ref{lem:precise_bound} will follow from Lemma~\ref{lem:conditional_circle_average} as follows. First, we have that Lemma~\ref{lem:conditional_circle_average} combined with the Gaussian tail bound allow us to bound the probability that
\begin{equation*}
M_{e^{-t}}(z) > 4\exp\left(\frac{2t - (h_{e^{-t}}(z)-h_1(z))}{A}\right)\,\,\text{and}\,\, 2t -(h_{e^{-t}}(z)-h_1(z)) \in [k,k+1]\,\,\text{for}\,\, k \in \NN.   
\end{equation*}
Then, Lemma~\ref{lem:precise_bound} will follow from a union bound over all $z \in (e^{-t-100} \ZZ^2) \cap U$, where $U \subseteq \CC$ is a fixed bounded open set. Finally, we will deduce Lemma~\ref{lem:main_estimate} from Lemma~\ref{lem:precise_bound} together with a union bound over all $n$.

We note that Lemmas~\ref{lem:conditional_circle_average} and ~\ref{lem:precise_bound} are stated in terms of $h_{e^{-t}}(z)-h_1(z)$ instead of $h_{e^{-t}}(z)$. The reason for this is that $h_{e^{-t}}(z)-h_1(z)$ is independent from $M_{e^{-t}}(z)$ due to Proposition~\ref{prop:conditional_independence} and so it makes it easier to obtain the required estimates. Then, the term $h_1(z)$ will be absorbed into the (random) constant $C$ appearing in the statement of Lemma~\ref{lem:main_estimate} since $\sup_{z \in U} |h_1(z)|$ is finite a.s.

\begin{lemma}\label{lem:conditional_circle_average}
There exist universal constants $c_0,c_1>0$ such that the following holds for all $S>1,z \in \CC$, and all $t > 0$. A.s., we have that
\begin{equation*}
    \mathbb{P}\left[M_{e^{-t}}(z) > S \giv h_{e^{-t}}(z) - h_1(z) \right] \leq c_0 (S/4)^{-b(\xi) \sqrt{M(\xi)}} + c_0 e^{-c_1 (\log S)^2}.
\end{equation*}
\end{lemma}

\begin{proof}
By the locality and the Weyl scaling properties of $\widetilde{D}_h^{\xi}$ (Axioms~\eqref{it:locality} and ~\eqref{it:weyl_scaling} respectively), we obtain that the random variable $M_{e^{-t}}(z)$ is a.s. determined by $h|_{B_{e^{-t}}}(z)$, viewed modulo additive constant, and so it is a.s. determined by $(h-h_{e^{-t}}(z))|_{B_{e^{-t}}(z)}$. Hence, combining with Proposition~\ref{prop:conditional_independence} and the translation invariance of the law of $h$, it follows that $M_{e^{-t}}(z)$ is independent from $h_{e^{-t}}(z)-h_1(z)$. Therefore, to complete the proof of the lemma, it suffices to prove a tail bound for the unconditional law of $M_{e^{-t}}(z)$.

By the translation invariance of the law of $h$, modulo additive constant, combined with Weyl scaling, we obtain that
\begin{align}\label{eq:equality_in_law}
    &\exp(-\xi h_{e^{-t}}(z) + \xi Q(\xi) t) \widetilde{D}_h^{\xi}(\text{around}\,\, \mathbb{A}_{e^{-t-51/100},e^{-t-1/2}}(z)) \stackrel{d}{=} \widetilde{D}_h^{\xi}(\text{around}\,\, \mathbb{A}_{e^{-51/100},e^{-1/2}}(0)),\nonumber \\
    &\exp(-\xi h_{e^{-t}}(z) + \xi Q(\xi) t) \widetilde{D}_h^{\xi}(\text{across}\,\, \mathbb{A}_{e^{-t-100},e^{-t}}(z)) \stackrel{d}{=} \widetilde{D}_h^{\xi}(\text{across}\,\, \mathbb{A}_{e^{-100},1}(0)).
\end{align}
Thus, combining with Proposition~\ref{prop:main_lqg_estimate}, we obtain that
\begin{equation}\label{eq:dist_across_not_big}
    \mathbb{P}[\exp(-\xi h_{e^{-t}}(z) + \xi Q(\xi) t) \widetilde{D}_h^{\xi}(\text{across}\,\, \mathbb{A}_{e^{-t-100},e^{-t}}(z))>S] = O_S(S^{-b(\xi) \sqrt{M(\xi)}})\quad \text{as} \quad S \to \infty,
\end{equation}
at a rate which is universal. 

Next, we consider the regions
\begin{align*}
   & G_1 = \mathbb{A}_{e^{-51/100},e^{-1/2}}(0) \cap \{w \in \CC : \Re(w) > -e^{-51/100}\}, G_2 = \mathbb{A}_{e^{-51/100},e^{-1/2}}(0) \cap \{w \in \CC : \Re(w) < e^{-51/100}\},\\
   &F_1 = \mathbb{A}_{e^{-51/100},e^{-1/2}}(0) \cap \{w \in \CC : -e^{-51/100}/2 <\Re(w) < e^{-51/100}/2, \Im(w)>0\},\\
   &F_2 = \mathbb{A}_{e^{-51/100},e^{-1/2}}(0) \cap \{w \in \CC : -e^{-51 / 100} / 2 < \Re(w)<e^{-51 / 100} / 2 , \Im(w) < 0\}.
\end{align*}
Then, Proposition~\ref{prop:main_lqg_estimate} implies that 
\begin{align*}
    \mathbb{P}[\widetilde{D}_h^{\xi}(K_{1,1},K_{2,1} ; G_1) \geq S] + \mathbb{P}[\widetilde{D}_h^{\xi}(K_{1,2},K_{2,2} ; G_2) \geq S] = O_S(S^{-b(\xi) \sqrt{M(\xi)}})\,\,\text{as} \,\, S 
    \to \infty
\end{align*}
and for $j=1,2$, we have
\begin{align*}
    \mathbb{P}[\widetilde{D}_h^{\xi}(\partial F_j \cap \partial B_{e^{-1/2}}(0) , \partial F_j \cap \partial B_{e^{-51/100}}(0) ; F_j) \geq S] = O_S(S^{-b(\xi) \sqrt{M(\xi)}})\,\,\text{as}\,\,S \to \infty,
\end{align*}
at a rate which is universal, where
\begin{align*}
    &K_{1,1} = \mathbb{A}_{e^{-51/100},e^{-1/2}}(0) \cap \{w \in \CC : \Re(w) = -e^{-51/100}/2 , \Im(w)>0\},\\
    &K_{2,1} = \mathbb{A}_{e^{-51/100},e^{-1/2}}(0) \cap \{w \in \CC : \Re(w) = -e^{-51/100}/2 , \Im(w)<0\},
\end{align*}
and
\begin{align*}
    &K_{1,2} = \mathbb{A}_{e^{-51/100},e^{-1/2}}(0) \cap \{w \in \CC : \Re(w) = e^{-51/100}/2 , \Im(w)>0\},\\
    &K_{2,2} = \mathbb{A}_{e^{-51/100},e^{-1/2}}(0) \cap \{w \in \CC : \Re(w) = e^{-51/100} /2, \Im(w)<0\}.
\end{align*}
Suppose that we are working on the event that
\begin{align*}
    &\widetilde{D}_h^{\xi}(K_{1,1},K_{2,1} ; G_1) < S,\,\,\widetilde{D}_h^{\xi}(K_{1,2},K_{2,2} ; G_2) < S,\\
    &\widetilde{D}_h^{\xi}(\partial F_j \cap \partial B_{e^{-51/100}}(0) , \partial F_j \cap \partial B_{e^{-1/2}}(0) ; F_j) < S \quad \text{for} \quad j=1,2.
\end{align*}
Then, for all $i=1,2$, there exist paths $P_i,Q_i$ in $G_i$ and $F_i$ respectively, such that $P_i$ (resp.\ $Q_i$) connects $K_{1,i}$ to $K_{2,i}$ (resp.\ $\partial F_i \cap \partial B_{e^{-51/100}}(0)$ to $\partial F_i \cap \partial B_{e^{-1/2}}(0)$), and such that $\text{len}(P_i ; \widetilde{D}_h^{\xi}) < S$ and $\text{len}(Q_i ; \widetilde{D}_h^{\xi}) < S$. Then, there exists a path $P$ contained in the concatenation of the paths $P_1,P_2,Q_1$, and $Q_2$, such that $P$ is a path in $\mathbb{A}_{e^{-51/100},e^{-1/2}}(0)$ disconnecting $\partial B_{e^{-51/100}}(0)$ from $\partial B_{e^{-1/2}}(0)$ and $\text{len}(P ; \widetilde{D}_h^{\xi}) < 4S$. Therefore, we obtain that
\begin{equation}\label{eq:dist_around_not_big}
\mathbb{P}[\widetilde{D}_h^{\xi}(\text{around}\,\,\mathbb{A}_{e^{-51/100},e^{-1/2}}(0)) > S] = O_S((S/4)^{-b(\xi) \sqrt{M(\xi)}}) \quad \text{as} \quad S \to \infty,   
\end{equation}
at a rate which is universal.

Finally, it follows from combining the scale and translation invariance of the law of $h$, viewed modulo additive constant, with the argument given in the proof of \cite[Lemma~3.3]{ding2024critical} (see in particular equation (3.8) in the proof of \cite[Lemma~3.3]{ding2024critical}), that there exist universal constants $b_0,b_1 \in (0,\infty)$ such that
\begin{equation}\label{eq:bound_on_circle_averages}
\mathbb{P}\left[\sup_{r \in [e^{-t-100},e^{-t}]} \sup_{w \in B_{e^{-t}-r}(z)} (|h_r(w) - h_{e^{-t}}(z)|) > \log S \right] \leq b_0 e^{-b_1 (\log S)^2} \,\, \text{for all}\,\,z \in \CC, t\geq 0, S>1.  
\end{equation}
Hence, the proof of the lemma is complete by combining \eqref{eq:equality_in_law}, ~\eqref{eq:dist_across_not_big}, ~\eqref{eq:dist_around_not_big} with \eqref{eq:bound_on_circle_averages}.    
\end{proof}

\begin{lemma}\label{lem:precise_bound}
    Fix a bounded open set $U \subseteq \CC$ and $\alpha , A \in (0,\infty)$. Suppose that we have chosen the functions $b,M$ such that $b(\xi) \sqrt{M(\xi)} > 2A$. Then, there exist constants $c_0,c_1 \in (0,\infty)$ depending only on $U,\alpha,$ and $A$ such that for all $t\geq 1$, it holds with probability at least 
    \begin{equation*}
        1-c_0 t^{1-\left(\frac{b(\xi) \sqrt{M(\xi)}}{A}-2\right)\alpha} - c_0 e^{-c_1 (\log t)^2}
    \end{equation*}
that the following is true. For all $z \in (e^{-t-100} \ZZ^2) \cap U$ such that $h_{e^{-t}}(z) - h_1(z) \leq 2t - \alpha \log t$, we have that
\begin{equation*}
    M_{e^{-t}}(z) \leq 4\exp\left(\frac{2t - (h_{e^{-t}}(z) - h_1(z))}{A}\right).
\end{equation*}
\end{lemma}

\begin{proof}
We will follow the argument in the proof of \cite[Lemma~3.5]{ding2024critical} by partitioning the set of points $z \in (e^{-t-100} \ZZ^2) \cap U$ for which $M_{e^{-t}}(z)$ is large.

Let $k_t:=\lfloor \alpha \log t\rfloor$ and $K_t:=\lceil 2t \rceil$, and note that if $0 \leq h_{e^{-t}}(z) - h_1(z) \leq 2t -\alpha \log t$, then $2t - (h_{e^{-t}}(z) - h_1(z))$ belongs to $[k,k+1]$ for some $k \in [k_t , K_t-1]_{\ZZ}$. For $k \in [k_t , K_t-1]_{\ZZ}$, we let $Z_t^k$ be the set of $z \in (e^{-t-100} \ZZ^2) \cap U$ such that
\begin{equation*}
    M_{e^{-t}}(z) > 4 \exp\left(\frac{2t - (h_{e^{-t}}(z)-h_1(z))}{A}\right)\,\,\text{and}\,\, 2t-(h_{e^{-t}}(z)-h_1(z)) \in [k,k+1].
\end{equation*}
Also, we let $Z_t^{K_t}$ be the set of $z \in (e^{-t-100} \ZZ^2) \cap U$ such that
\begin{equation*}
    M_{e^{-t}}(z) > 4 \exp\left(\frac{2t-(h_{e^{-t}}(z)-h_1(z))}{A}\right) \quad \text{and} \quad h_{e^{-t}}(z) - h_1(z) \leq 0.
\end{equation*}

First, we we will prove that there exist constants $c_0 , c_1 \in (0,\infty)$ depending only on $U,\alpha$, and $A$ such that
\begin{equation}\label{eq:bounding_the_sum}
\mathbb{E}\left[\sum_{k=k_t}^{K_t} \# Z_t^k \right] \leq c_0 t^{1-\left(\frac{b(\xi) \sqrt{M(\xi)}}{A}-2\right)\alpha} + c_0 e^{-c_1 (\log t)^2} \quad \text{for all} \quad t \geq 1.   
\end{equation}
Lemma~\ref{lem:conditional_circle_average} implies that there exist universal constants $c_0,c_1 \in (0,\infty)$ such that for all $z \in \CC,k \in \NN$, and all $x \in [k,k+1]$, we have that
\begin{equation}\label{eq:conditioning_bound}
\mathbb{P}[M_{e^{-t}}(z) > 4\exp(x/A) \giv 2t - (h_{e^{-t}}(z) - h_1(z)) =x] \leq c_0 \exp\left(-\left(\frac{b(\xi) \sqrt{M(\xi)}}{A}\right)x\right) + c_0 \exp\left(-\frac{c_1 x^2}{A^2}\right). 
\end{equation}
Also, the random variable $h_{e^{-t}}(z)-h_1(z)$ is a centered Gaussian with variance $t$, and so if $k \in [k_t,K_t]_{\ZZ}$, we have that
\begin{align}\label{eq:gaussian_tail_bound}
\mathbb{P}[2t - (h_{e^{-t}}(z) - h_1(z)) \leq k+1] &= \mathbb{P}[h_{e^{-t}}(z) - h_1(z) \geq 2t - (k+1)]\nonumber\\
&\leq \exp\left(-\frac{(2t-k-1)^2}{2t}\right)\nonumber\\
&\leq \exp(-2t + 2(k+1)).   
\end{align}

Combining \eqref{eq:conditioning_bound} with \eqref{eq:gaussian_tail_bound}, we obtain that for all $z \in (e^{-t-100} \ZZ^2) \cap U$ and all $k \in [k_t,K_t-1]_{\ZZ}$, we have that
\begin{equation*}
    \mathbb{P}[z \in Z_t^k] \leq c_0 \exp\left(-\left(\frac{b(\xi) \sqrt{M(\xi)}}{A}-2\right) k -2t\right) + c_0 \exp\left(-\left(\frac{c_1}{A^2}\right) k^2 -2t + 2k\right),
\end{equation*}
by possibly taking $c_0$ to be larger (in a universal way). Moreover, we have by Lemma~\ref{lem:conditional_circle_average} that
\begin{align*}
    \mathbb{P}[z \in Z_t^{K_t}] \leq \mathbb{P}[M_{e^{-t}}(z) > 4\exp(2t / A)] \leq c_0 \exp\left(-\left(\frac{2b(\xi) \sqrt{M(\xi)}}{A}\right) t\right) + c_0 \exp\left(-\left(\frac{4c_1}{A^2}\right) t^2 \right).
\end{align*}
Therefore, by taking a union bound over $O_t(e^{2t})$ points in $(e^{-t-100} \ZZ^2) \cap U$ and by possibly taking $c_0$ to be larger and $c_1>0$ to be smaller (in a way that depends only on $U$ and $A$), we get that
\begin{equation*}
    \mathbb{E}[\#Z_t^k] \leq c_0 \exp\left(-\left(\frac{b(\xi) \sqrt{M(\xi)}}{A}-2\right)k \right) + c_0 \exp\left(-\left(\frac{c_1}{A^2}\right)k^2 + 2k \right)
\end{equation*}
for all $k \in [k_t,K_t-1]_{\ZZ}, t \geq 1$ and 
\begin{equation*}
    \mathbb{E}[\# Z_t^{K_t}] \leq c_0 \exp\left(-\left(\frac{b(\xi) \sqrt{M(\xi)}}{A}-2\right) t \right) + c_0 \exp\left(-\left(\frac{c_1}{A^2}\right) K_t^2 \right).
\end{equation*}

Hence, 
\begin{align}\label{eq:bounding_sum_1}
\mathbb{E}\left[\sum_{k=k_t}^{K_t} \#Z_t^k \right] &\leq c_0 \sum_{k=k_t}^{K_t-1} \left(\exp\left(-\left(\frac{b(\xi)\sqrt{M(\xi)}}{A}-2\right) k\right) + \exp\left(-\left(\frac{c_1}{A^2}\right)k^2 + 2k\right)\right)\notag\\
&+c_0 \exp\left(-\left(\frac{b(\xi) \sqrt{M(\xi)}}{A}-2\right)t \right) + c_0 \exp\left(-\left(\frac{c_1}{A^2}\right)K_t^2\right).
\end{align}
for all $t \geq 1$.

Suppose that we choose the functions $b,M$ such that $b(\xi) \sqrt{M(\xi)} > 2A$. Then, it holds that
\begin{equation*}
    \exp\left(-\left(\frac{b(\xi) \sqrt{M(\xi)}}{A}-2\right)k\right)\leq \exp\left(-\left(\frac{b(\xi) \sqrt{M(\xi)}}{A}-2\right)\alpha \log t\right)
\end{equation*}
for all $k \in [k_t , K_t-1]_{\ZZ}, t \geq 1$ and
\begin{equation*}
    \exp\left(-\left(\frac{b(\xi) \sqrt{M(\xi)}}{A}-2\right)t \right) \leq \exp\left(-\left(\frac{b(\xi)\sqrt{M(\xi)}}{A}-2\right) \alpha \log t \right)
\end{equation*}
for all $t \geq t_0$, where $t_0 > 1$ depends only on $\alpha$.
Moreover, there exist constants $\widetilde{c}_0,\widetilde{c}_1 \in (0,\infty)$ depending only on $U,\alpha$, and $A$, such that
\begin{equation*}
    c_0 \exp\left(-\left(\frac{c_1}{A^2}\right) k^2 + 2k\right) \leq \widetilde{c}_0 \exp(-\widetilde{c}_1 k^2) \,\,\text{for all} \,\, k \in [k_t,K_t]_{\ZZ}, t \geq 1.
\end{equation*}
Therefore, combining with \eqref{eq:bounding_sum_1} and by possibly taking $\widetilde{c}_0$ to be larger (in a way that depends only on $U,\alpha$, and $A$), we obtain that
\begin{equation*}
    \mathbb{E}\left[\sum_{k=k_t}^{K_t} \#Z_t^k \right] \leq \widetilde{c}_0 t \exp\left(-\left(\frac{b(\xi) \sqrt{M(\xi)}}{A}-2\right) \alpha \log t \right) + \widetilde{c}_0 t \exp(-\widetilde{c}_1 (\log t)^2) \,\,\text{for all} \,\,t \geq 1.
\end{equation*}
This completes the proof of \eqref{eq:bounding_the_sum}.

Finally, we note that \eqref{eq:bounding_the_sum} implies that 
\begin{equation*}
    \mathbb{P}\left[\bigcup_{k=k_t}^{K_t} Z_t^k \neq \emptyset \right] \leq \mathbb{E}\left[\sum_{k=k_t}^{K_t} \#Z_t^k \right] \leq c_0 t^{1-\left(\frac{b(\xi) \sqrt{M(\xi)}}{A}-2\right)\alpha} + c_0 e^{-c_1 (\log t)^2} \,\, \text{for all}\,\, t \geq 1.
\end{equation*}
This completes the proof of the lemma since we have that
\begin{equation*}
    M_{e^{-t}}(z) \leq 4\exp\left(\frac{2t - (h_{e^{-t}}(z) - h_1(z))}{A}\right)
\end{equation*}
for all $z \in (e^{-t-100} \ZZ^2) \cap U$ such that $h_{e^{-t}}(z) - h_1(z) \leq 2t - \alpha \log t$ if $\bigcup_{k=k_t}^{K_t} Z_t^k = \emptyset$.  
\end{proof}

Now we are ready to prove Lemma~\ref{lem:main_estimate}.

\begin{proof}[Proof of Lemma~\ref{lem:main_estimate}.]
Since $U$ is bounded and $z \mapsto h_1(z)$ is continuous, we have that
\begin{equation*}
    C_0:=\sup_{z \in U} |h_1(z)| < \infty \quad \text{a.s.}
\end{equation*}
Fix $1/5 < \alpha < \alpha' < 1/4$ deterministic. Then, Proposition~\ref{prop:circle_average_bound} implies that it is a.s. the case that there exists $n_0 \in \NN$ such that for all $n \geq n_0$, we have that
\begin{equation*}
    |h_{e^{-n}}(z)| \leq 2n -\alpha' \log n \quad \text{for all} \quad z \in (e^{-n-100} \ZZ^2) \cap U.
\end{equation*}
Since $(\alpha'-\alpha) \log n > C_0$ for all $n \in \NN$ sufficiently large, we obtain that it is a.s. the case that there exists $n_1 \in \NN, n_1 \geq n_0$, such that
\begin{equation*}
    |h_{e^{-n}}(z)| \leq 2n -\alpha \log n - C_0 \quad \text{for all} \quad z \in (e^{-n-100} \ZZ^2) \cap U.
\end{equation*}
It follows that there exist deterministic constants $C_0 \in (0,\infty), n_0 \in \NN$ depending only on $\alpha,\varepsilon$, and $U$, such that the following holds with probability at least $1-\varepsilon / 100$.
\begin{equation}\label{eq:gff_bounds}
    \sup_{z \in U} |h_1(z)| \leq C_0,\,\,|h_{e^{-n}}(z)| \leq 2n -\alpha \log n - C_0\,\,\text{for all} \,\, n \geq n_0, z \in (e^{-n-100} \ZZ^2) \cap U.
\end{equation}

Next, for all $n \in \NN$, we let $E_n$ denote the event that
\begin{equation*}
    M_{e^{-n}}(z) \leq 4\exp\left(\frac{2n - (h_{e^{-n}}(z)-h_1(z))}{A}\right)
\end{equation*}
for all $z \in (e^{-n-100} \ZZ^2) \cap U$ with $h_{e^{-n}}(z) - h_1(z) \leq 2n - \alpha \log n$. Then, Lemma~\ref{lem:precise_bound} implies that there exist constants $c_0,c_1 \in (0,\infty)$ depending only on $U,\alpha$, and $A$ such that
\begin{equation*}
    \mathbb{P}[E_n] \geq 1 - c_0 n^{1-\left(\frac{b(\xi) \sqrt{M(\xi)}}{A}-2\right)\alpha } - c_0 \exp(-c_1 (\log n)^2).
\end{equation*}
Note that
\begin{equation*}
    \left(\frac{b(\xi)\sqrt{M(\xi)}}{A}-2\right)\alpha - 1 > 2
\end{equation*}
and so 
\begin{equation*}
    \mathbb{P}[E_n^c] \leq c_0 n^{-2} + c_0 \exp(-c_1 (\log n)^2) \quad \text{for all} \quad n \in \NN,
\end{equation*}
which implies that $\sum_{n \in \NN} \mathbb{P}[E_n^c] \lesssim 1$, where the implicit constant depends only on $U,\alpha$, and $A$. Therefore, there exists $n_1 \in \NN, n_1 \geq n_0$, depending only on $\varepsilon,U,\alpha$, and $A$, such that
\begin{equation*}
    \mathbb{P}\left[\bigcap_{n=n_1}^{\infty} E_n \right] \geq 1 - \frac{\varepsilon}{100}.
\end{equation*}
Thus, combining with \eqref{eq:gff_bounds}, we obtain that both \eqref{eq:gff_bounds} and $\cap_{n=n_1}^{\infty} E_n$ occur with probability at least $1 - \varepsilon/50$. For the rest of the proof, we will assume that both \eqref{eq:gff_bounds} and $\cap_{n=n_1}^{\infty} E_n$ occur.

Fix $n \in \NN, n \geq n_1$ and $z \in (e^{-n-100} \ZZ^2) \cap U$. Then, since \eqref{eq:gff_bounds} and $E_n$ both occur, we have that
\begin{equation*}
    h_{e^{-n}}(z) - h_1(z) \leq 2n - \alpha \log n
\end{equation*}
and 
\begin{align}\label{eq:bound_large_n}
M_{e^{-n}}(z) \leq 4\exp\left(\frac{2n - (h_{e^{-n}}(z) - h_1(z))}{A}\right)&\leq 4\exp\left(\frac{2n - h_{e^{-n}}(z) + C_0}{A}\right)\nonumber\\
& = 4e^{C_0 / A} \exp\left(\frac{2n - h_{e^{-n}}(z)}{A}\right).   
\end{align}
Moreover, since $b(\xi) \sqrt{M(\xi)}>1$, Lemma~\ref{lem:conditional_circle_average} implies that there exist universal constants $\widetilde{c}_0,\widetilde{c}_1 \in (0,\infty)$ such that
\begin{equation*}
    \mathbb{P}[M_{e^{-n}}(z) > S] \leq \widetilde{c}_0 S^{-1} + \widetilde{c}_0 e^{-\widetilde{c}_1 (\log S)^2}\,\,\text{for all}\,\,S>1,n \in \NN, z \in (e^{-n-100} \ZZ^2) \cap U.
\end{equation*}
Thus, there exists $C_1 \in (0,\infty)$ depending only on $\varepsilon$ and $U$ such that with probability at least $1 - \varepsilon/100$, it holds that
\begin{equation*}
    M_{e^{-n}}(z) \leq C_1\,\,\text{for all} \,\, 1 \leq n \leq n_1, z \in (e^{-n-100}\ZZ^2) \cap U.
\end{equation*}

Furthermore, there exists $C_2 \in (0,\infty)$ depending only on $\varepsilon,A$ and $U$ such that with probability at least $1-\varepsilon/100$, we have that
\begin{equation*}
    \exp\left(\frac{2n - h_{e^{-n}}(z)}{A}\right) \geq C_2 \,\,\text{for all}\,\, 1 \leq n \leq n_1, z \in (e^{-n-100}\ZZ^2) \cap U.
\end{equation*}
In particular, with probability at least $1-\varepsilon/50$, it holds that
\begin{equation}\label{eq:small_n'}
    M_{e^{-n}}(z) \leq \left(\frac{C_1}{C_2}\right) \exp\left(\frac{2n - h_{e^{-n}}(z)}{A}\right)\,\,\text{for all}\,\, 1 \leq n \leq n_1, z \in (e^{-n-100}\ZZ^2) \cap U.
\end{equation}
Therefore, we conclude the proof of the lemma by combining \eqref{eq:bound_large_n} with \eqref{eq:small_n'}.  
\end{proof}

\subsection{Stringing together paths}
\label{subsec:stringing_paths}

In this subsection, we are going to use Lemma~\ref{lem:main_estimate} to prove the following estimate which is going to play a crucial role in the proof of \eqref{eq:sum_of_distances}.

\begin{lemma}\label{lem:stringing_paths}
Fix a bounded open set $U \subseteq \CC$ and let $\varepsilon \in (0,1), \alpha \in (0,1/4), A \in (2,\infty)$ be such that $\alpha (1+2/A) < 1/4$. Suppose that the functions $b,M$ are chosen such that $b(\xi) \sqrt{M(\xi)} > 20 A$. Then, there exists a constant $C \in (0,\infty)$ depending only on $\varepsilon,\alpha,A$, and $U$, such that the following hold with probability at least $1-\varepsilon$. For all $z \in U$, we have that
\begin{enumerate}
    \item \label{it:average_under_control}
    \begin{equation*}
        h_{e^{-t}}(z) \leq 2t - \alpha \log t + C\,\,\text{for all} \,\, t \geq 0.
    \end{equation*}
    \item \label{it:across_under_control}
    For all $n,m \in \NN$ with $m<n$, we have
    \begin{align*}
        &\widetilde{D}_h^{\xi}(\partial B_{e^{-n-50}}(z),\partial B_{e^{-m}}(z)) \\
        &\leq C^{\xi +1} e^{C(\xi+1) / A} e^{2\xi Q(\xi)} \sum_{k=m}^{n-1} \min_{t \in [k,k+1]}\left(\exp\left(\xi h_{e^{-t}}(z) - \xi Q(\xi) t + \frac{(2+2\xi) (2t - h_{e^{-t}}(z))}{A}\right)\right).
    \end{align*}
    \item \label{it:around_under_control}
    For all $n \in \NN$, we have
    \begin{align*}
        &\widetilde{D}_h^{\xi}(\text{around}\,\, \mathbb{A}_{e^{-n-1},e^{-n}}(z))\\
        &\leq C^{\xi+1} e^{C(\xi+1)/A} e^{2\xi Q(\xi)} \min_{t \in [n+1,n+2]} \exp\left(\xi h_{e^{-t}}(z) - \xi Q(\xi) t + \frac{(2+2\xi) (2t - h_{e^{-t}}(z))}{A}\right).
    \end{align*}
\end{enumerate}
    
\end{lemma}

The proof of Lemma~\ref{lem:stringing_paths} will follow from an argument which is similar to that in the proof of \cite[Lemma~3.6]{ding2024critical}. However, as in the proof of Lemma~\ref{lem:main_estimate}, we will need to make sure that the parameters in the relevant estimates are uniform in the parameter $\xi>0$ (provided we choose the functions $b,M$ appropriately). This is the main reason for having a different statement in Lemma~\ref{lem:stringing_paths} than that of \cite[Lemma~3.6]{ding2024critical} with the term $(2t - h_{e^{-t}}(z)) / A$ in place of the term $|2t - h_{e^{-t}}(z)|^{1/2 + \zeta}$ for some $\zeta>0$, since the current statement of Lemma~\ref{lem:stringing_paths} makes it easier to prove that the required estimates are uniform in $\xi>0$.

Item~\eqref{it:across_under_control} in Lemma~\ref{lem:stringing_paths} will allow us to bound distances between different scales in terms of the circle average process $(z,r) \mapsto h_r(z)$ while item~\eqref{it:around_under_control} will allow us to string paths between pairs of circles whose centers can be different.

Before we proceed with the proof of Lemma~\ref{lem:stringing_paths}, we state some consequences of Proposition~\ref{prop:circle_average_bound} and Lemma~\ref{lem:main_estimate} that we are going to use. In particular, we fix $\alpha \in (0,1/4), A \in (2,\infty),\varepsilon \in (0,1)$ such that $\alpha (1+2/A) < 1/4$, and let $U \subseteq \CC$ be a bounded open set. Then, Lemma~\ref{lem:main_estimate} implies that there exists a constant $C_0 \in (0,\infty)$ depending only on $\varepsilon,A$, and $U$ such that with probability at least $1-\varepsilon/200$, we have that
\begin{equation}\label{eq:main_condition}
    M_{e^{-n}}(z) < C_0 \exp\left(\frac{2n - h_{e^{-n}}(z)}{A}\right)\,\,\text{for all} \,\, n \in \NN, z \in (e^{-n-100} \ZZ^2) \cap B_1(U).
\end{equation}

Suppose that \eqref{eq:main_condition} holds. Then, we have that the following hold for all $n \in \NN, z \in (e^{-n-100} \ZZ^2) \cap B_1(U)$.
\begin{enumerate}
    \item \label{it:path_around}
    There exists a path $P_n(z)$ in the annulus $\mathbb{A}_{e^{-n-51/100},e^{-n-1/2}}(z)$ which disconnects the inner from the outer boundaries of the annulus and has $\widetilde{D}_h^{\xi}$-length at most
    \begin{equation*}
        \exp(\xi h_{e^{-n}}(z) - \xi Q(\xi) n) M_{e^{-n}}(z) \leq C_0 \exp\left(\xi h_{e^{-n}}(z) - \xi Q(\xi) n + \left(\frac{2n-h_{e^{-n}}(z)}{A}\right)\right).
    \end{equation*}
    \item \label{it:path_across}
    There exists a path $Q_n(z)$ between the inner and outer boundaries of $\mathbb{A}_{e^{-n-100},e^{-n}}(z)$, whose $\widetilde{D}_h^{\xi}$-length is at most
    \begin{equation*}
        C_0 \exp\left(\xi h_{e^{-n}}(z) - \xi Q(\xi) n + \left(\frac{2n - h_{e^{-n}}(z)}{A}\right)\right).
    \end{equation*}
    \item \label{it:circle_averages_bound}
    We have
    \begin{equation*}
        \sup_{r \in [e^{-n-100},e^{-n}]} \sup_{w \in B_{e^{-n}-r}(z)} |h_r(w) - h_{e^{-n}}(z)| \leq \log C_0 + \left(\frac{2n-h_{e^{-n}}(z)}{A}\right).
    \end{equation*}   
\end{enumerate}

Fix $\alpha_0 \in (\alpha , 1/4)$ such that $\alpha (1+2/A) < \alpha_0$ and note that we can always choose such $\alpha_0$ since we have assumed that $\alpha (1+2/A) < 1/4$. Then, Proposition~\ref{prop:circle_average_bound} implies that by possibly taking $C_0$ to be larger (in a way that depends only on $\varepsilon,\alpha_0,A$, and $U$), we have with probability at least $1-\varepsilon/200$ that
\begin{equation}\label{eq:circle_ave_ineq}
    |h_{e^{-n}}(z)| \leq 2n -\alpha_0 \log n + \log C_0 \,\,\text{for all} \,\, n \in \NN, z \in (e^{-n-100} \ZZ^2) \cap B_1(U).
\end{equation}

Note that the above estimates can be applied only for the case that the annuli $\mathbb{A}_{e^{-n-51/100},e^{-n-1/2}}(z)$ and $\mathbb{A}_{e^{-n-100},e^{-n}}(z)$ are centered at points on the grid $(e^{-n-100} \ZZ^2) \cap B_1(U)$. To prove Lemma~\ref{lem:stringing_paths}, we need to extend the above estimates to the case that the above annuli are centered at \emph{any} point $z \in U$. This is the purpose of the following lemma.

\begin{lemma}\label{lem:uniform_circle_ave}
There exists a constant $C_1 \in (0,\infty)$ depending only on $\alpha,\alpha_0,\varepsilon,A$, and $U$ such that the following holds with probability at least $1-\varepsilon/50$. Fix $z \in U$. For $n \in \NN$, we let $z_n \in (e^{-n-100} \ZZ^2) \cap B_1(U)$ be chosen such that $z \in B_{e^{-n-99}}(z_n)$. Then, for all $t \in [n+1,n+2]$, we have that
\begin{equation*}
    h_{e^{-t}}(z) \leq h_{e^{-n}}(z_n) + C_1 + \left(\frac{2(2t - h_{e^{-t}}(z))}{A}\right)
\end{equation*}
 and
\begin{equation*}
    2n - h_{e^{-n}}(z_n) \leq 2(2t - h_{e^{-t}}(z)) + C_1.
\end{equation*}
\end{lemma}

\begin{proof}
Suppose that items~\eqref{it:path_around}-\eqref{it:circle_averages_bound}, and \eqref{eq:circle_ave_ineq} hold, and note that as explained before the statement of the lemma, this occurs with probability at least $1-\varepsilon/100$. 

Let $n \in \NN, t \in [n+1,n+2]$. By item~\eqref{it:circle_averages_bound} applied with $r=e^{-t}, z_n$ in place of $z$, and $z$ in place of $w$, we get that
\begin{equation*}
    |h_{e^{-t}}(z)  - h_{e^{-n}}(z_n)| \leq \left(\frac{2n - h_{e^{-n}}(z_n)}{A}\right) + \log C_0.
\end{equation*}
Thus, it follows that
\begin{align*}
    2n - h_{e^{-n}}(z_n) &= (2n-h_{e^{-t}}(z)) + (h_{e^{-t}}(z) - h_{e^{-n}}(z_n))\\
    &\leq (2t - h_{e^{-t}}(z)) + \log C_0 + \left(\frac{2n - h_{e^{-n}}(z_n)}{A}\right),
\end{align*}
which implies that
\begin{equation*}
    \left(1-\frac{1}{A}\right) (2n - h_{e^{-n}}(z_n)) \leq (2t - h_{e^{-t}}(z)) + \log C_0.
\end{equation*}
Moreover, by \eqref{eq:circle_ave_ineq}, we obtain that there exists deterministic $n_\star \in \NN$ depending only on $\alpha_0,A$, and $C_0$, such that for all $n \geq n_\star$,
\begin{equation*}
    \left(1-\frac{1}{A}\right) (2n - h_{e^{-n}}(z_n)) - \log C_0 \geq \left(1-\frac{1}{A}\right) (\alpha_0 \log n - \log C_0) - \log C_0 > 0.
\end{equation*}
In particular, $2t - h_{e^{-t}}(z) > 0$ if $n \geq n_\star$, and so 
\begin{align}\label{eq:big_n}
    2n - h_{e^{-n}}(z_n) &\leq \left(\frac{A}{A-1}\right) \left(2t - h_{e^{-t}}(z) + \log C_0\right)\nonumber\\
    &\leq 2(2t - h_{e^{-t}}(z)) + \left(\frac{A \log C_0}{A-1}\right)
\end{align}
if $n \geq n_\star$. Note that since the process $(w,r) \mapsto h_r(w)$ is continuous a.s., we get that there exists a deterministic constant $C_1 \in (\log C_0 / (1-1/A) , \infty)$ depending only on $\varepsilon,A,C_0,n_\star$, and $U$, such that with probability at least $1-\varepsilon/100$, we have that for all $1 \leq n \leq n_\star, t \in [n+1,n+2],z \in B_1(U)$,
\begin{equation}\label{eq:small_n}
    2n - h_{e^{-n}}(z_n) \leq 2(2t - h_{e^{-t}}(z)) + C_1.
\end{equation}

It follows from combining \eqref{eq:big_n} with \eqref{eq:small_n} that with probability at least $1-\varepsilon/50$, it holds that
\begin{equation}\label{eq:all_n}
2n -h_{e^{-n}}(z_n) \leq 2(2t - h_{e^{-t}}(z)) + C_1\,\,\text{for all} \,\, n \in \NN, z \in B_1(U), t \in [n+1,n+2].
\end{equation}
Suppose that \eqref{eq:all_n} holds. Then, combining with item~\eqref{it:circle_averages_bound}, we obtain that
\begin{align*}
    h_{e^{-t}}(z) &\leq h_{e^{-n}}(z_n) + \left(\frac{2n-h_{e^{-n}}(z_n)}{A}\right) + \log C_0\\
    &\leq h_{e^{-n}}(z_n) + \frac{2(2t - h_{e^{-t}}(z))}{A} + \frac{C_1}{A} + \log C_0.
\end{align*}
This completes the proof of the lemma.
\end{proof}

With Lemma~\ref{lem:uniform_circle_ave} at hand, we are ready to prove Lemma~\ref{lem:stringing_paths}.

\begin{proof}[Proof of Lemma~\ref{lem:stringing_paths}.]
\emph{Proof of \eqref{it:average_under_control}.} Combining Lemmas~\ref{lem:main_estimate} and ~\ref{lem:uniform_circle_ave} and Proposition~\ref{prop:circle_average_bound}, we obtain that there exist constants $C_0,C_1 \in (0,\infty)$ depending only on $\alpha,\alpha_0,\varepsilon,A$, and $U$, such that the following holds with probability at least $1-\varepsilon/25$. Items~\eqref{it:path_around}-\eqref{it:circle_averages_bound}, and \eqref{eq:circle_ave_ineq} hold, and the condition in the statement of Lemma~\ref{lem:uniform_circle_ave} holds as well. For the rest of the proof, we will assume that all the above conditions hold simultaneously.

Fix $z \in U$ and $n \in \NN$, and let $z_n \in (e^{-n-100} \ZZ^2) \cap B_1(U)$ be as in Lemma~\ref{lem:uniform_circle_ave}. Then, \eqref{eq:circle_ave_ineq} combined with the condition in the statement of Lemma~\ref{lem:uniform_circle_ave} imply that if $t \in [n+1,n+2]$ is such that $h_{e^{-t}}(z) \geq 2t - \alpha \log t$, then
\begin{align*}
h_{e^{-t}}(z)&\leq h_{e^{-n}}(z_n) + +C_1 + \left(\frac{2(2t-h_{e^{-t}}(z))}{A}\right) \leq 2n -\alpha_0 \log n + \log C_0 + \left(\frac{2(2t-h_{e^{-t}}(z))}{A}\right) + C_1\\
&\leq 2n - \alpha_0 \log n + \log C_0 + C_1 + \left(\frac{2\alpha}{A}\right) \log( n+2).  
\end{align*}
Since $\alpha (1+2/A) < \alpha_0$, we obtain that there exists deterministic $n_\star \in \NN$ depending only on $C_0,C_1,\alpha,\alpha_0$, and $A$, such that
\begin{equation*}
    -\alpha_0 \log n + \left(\frac{2\alpha}{A}\right) \log (n+2) < -\alpha \log n \quad \text{for all} \quad n \geq n_\star,
\end{equation*}
and hence
\begin{equation*}
    h_{e^{-t}}(z) \leq 2n -\alpha \log n + C_1 + \log C_0 \quad \text{for all} \quad n \geq n_\star.
\end{equation*}

Note that by the continuity of the circle average process, the quantity
\begin{equation*}
    \sup_{z \in U} \sup_{t \in [0,n_\star +1]} |h_{e^{-t}}(z)|
\end{equation*}
is a.s. finite. It follows that there exists constant $C>0$ depending only on $\varepsilon,\alpha,\alpha_0,A$, and $U$, such that with probability at least $1-\varepsilon/20$, we have that
\begin{equation*}
    h_{e^{-t}}(z) \leq 2t - \alpha \log t + C \quad \text{for all} \quad t \geq 0,\,\, z \in U.
\end{equation*}
This completes the proof of \eqref{it:average_under_control}.

\emph{Proof of \eqref{it:around_under_control}.} By Lemma~\ref{lem:uniform_circle_ave} and items~\eqref{it:path_around}-\eqref{it:circle_averages_bound}, we get that 
there exists deterministic constant $C_1>0$ depending only on $\varepsilon,\alpha,\alpha_0,A$, and $U$, such that the following holds with probability at least $1-\varepsilon/20$. For all $t \in [n+1,n+2]$, the $\widetilde{D}_h^{\xi}$-lengths of the paths $P_n(z)$ and $Q_n(z)$ from items~\eqref{it:path_around} and ~\eqref{it:path_across} are each bounded above by
\begin{align}\label{eq:ineq}
    &C_0 \exp\left(\xi h_{e^{-n}}(z_n) -\xi Q(\xi) n + \left(\frac{2n - h_{e^{-n}}(z_n)}{A}\right)\right)\nonumber\\
    &\leq C_0 \exp\left(\xi h_{e^{-t}}(z) -\xi Q(\xi) n + (\xi+1) \left(\frac{2n - h_{e^{-n}}(z_n)}{A}\right) + \xi \log C_0\right)\nonumber\\
    &\leq C_0 \exp\left(\xi h_{e^{-t}}(z) - \xi Q(\xi) t +2\xi Q(\xi) + 2(\xi + 1) \left(\frac{2t - h_{e^{-t}}(z)}{A}\right) + \left(\frac{(\xi+1)C_1}{A}\right) + \xi \log C_0\right)\nonumber\\
    &\leq C_0^{\xi +1} e^{C_1 (\xi +1)/A} e^{2\xi Q(\xi)} \exp\left(\xi h_{e^{-t}}(z) - \xi Q(\xi) t + \left(\frac{(2+2\xi) (2t - h_{e^{-t}}(z))}{A}\right)\right).
\end{align}
Note that for $n \in \NN$, the path $P_n(z)$ disconnects the inner and outer boundaries of $\mathbb{A}_{e^{-n-51/100},e^{-n-1/2}}(z_n)$. Since $|z-z_n| \leq e^{-n-99}$, this also disconnects the inner and outer boundaries of $\mathbb{A}_{e^{-n-1},e^{-n}}(z)$. Thus, \eqref{it:around_under_control} follows from \eqref{eq:ineq}. 

\emph{Proof of \eqref{it:across_under_control}.} Finally, we prove item~\eqref{it:across_under_control}. Let $n,m \in \NN$ be such that $m<n$. For all $k \in \{m-1,\cdots,n-2\}$, the path $Q_k(z_k)$ from condition \eqref{it:path_across} goes from $\partial B_{e^{-k-100}}(z_k)$ to $\partial B_{e^{-k}}(z_k)$. Note that
\begin{equation*}
    |z_{k+1}-z_k| \leq |z_{k+1}-z| + |z_k-z| \leq 2 e^{-k-99},
\end{equation*}
which implies that
\begin{equation*}
    \mathbb{A}_{e^{-k-51/100},e^{-k-1/2}}(z_k) \cup \mathbb{A}_{e^{-k-1-51/100},e^{-k-1-1/2}}(z_{k+1}) \subseteq \mathbb{A}_{e^{-k-100},e^{-k}}(z_k)
\end{equation*}
and both $\mathbb{A}_{e^{-k-51/100},e^{-k-1/2}}(z_k)$ and $\mathbb{A}_{e^{-k-1-51/100},e^{-k-1-1/2}}(z_{k+1})$ disconnect the inner and outer boundaries of $\mathbb{A}_{e^{-k-100},e^{-k}}(z_k)$. Hence, the path $Q_k(z_k)$ intersects $P_k(z_k)$ and $P_{k+1}(z_{k+1})$. Therefore, the union of the paths $Q_k(z_k)$ and $P_k(z_k)$ for $k \in \{m-1,\cdots,n-2\}$ is connected. 

Furthermore, since $|z-z_{m-1}| \leq e^{-m-98}$ and $|z-z_{n-2}| \leq e^{-n-97}$, we obtain that the path $Q_{m-1}(z_{m-1})$ intersects $\partial B_{e^{-m}}(z)$ and the path $Q_{n-2}(z_{n-2})$ intersects $\partial B_{e^{-n-50}}(z)$. Hence, summing \eqref{eq:ineq} gives that
\begin{align*}
    &\widetilde{D}_h^{\xi}(\partial B_{e^{-n-50}}(z) , \partial B_{e^{-m}}(z))\\
    &\leq 2 C_0^{\xi+1} e^{C_1 (\xi+1) / A} e^{2\xi Q(\xi)}\sum_{k=m-1}^{n-2} \min_{t \in [k+1,k+2]} \left(\exp\left(\xi h_{e^{-t}}(z) - \xi Q(\xi) t + \left(\frac{(2+2\xi)(2t - h_{e^{-t}}(z))}{A}\right)\right)\right)\\
    &=2C_0^{\xi+1} e^{C_1 (\xi+1) / A} e^{2\xi Q(\xi)}\sum_{k=m}^{n-1} \min_{t \in [k,k+1]} \left(\exp\left(\xi h_{e^{-t}}(z) - \xi Q(\xi) t + \left(\frac{(2+2\xi)(2t - h_{e^{-t}}(z))}{A}\right)\right)\right).  
\end{align*}
This completes the proof of \eqref{it:across_under_control} and hence the proof of the lemma.
\end{proof}

Now, we conclude this subsection by stating and proving Lemmas~\ref{lem:stringing_paths_1} and ~\ref{lem:around_poly} below. The former is a simplified version of items~\eqref{it:across_under_control} and ~\eqref{it:around_under_control} in Lemma~\ref{lem:stringing_paths} while the latter is a simplified version of item~\eqref{it:around_under_control_1} in Lemma~\ref{lem:stringing_paths_1}. For the rest of the paper in Sections~\ref{subsec:distances_between_scales} and ~\ref{subsec:proofs_of_main_results}, we will use Lemmas~\ref{lem:stringing_paths_1} and ~\ref{lem:around_poly} in place of Lemma~\ref{lem:stringing_paths} since the estimates in Lemmas~\ref{lem:stringing_paths_1} and ~\ref{lem:around_poly} make our proofs simpler.

\begin{lemma}\label{lem:stringing_paths_1}
Let $U \subseteq \CC$ be a bounded open set and fix $\varepsilon \in (0,1), \delta \in (0,\xi_c / 2)$. Suppose that we choose the continuous functions $b,M$ such that $b(\xi) \sqrt{M(\xi)} > 40 (1+\xi_c) / \delta$ for all $\xi \in (\xi_c / 2 , \xi_c]$. Then, the following holds for all $\xi \in (\xi_c/2 , \xi_c]$. There exists deterministic constant $C>1$ depending only on $\varepsilon,\delta$, and $U$, such that the following is true with probability at least $1-\varepsilon$. Fix $z \in U$. Then, we have the following.
\begin{enumerate}
    \item \label{it:across_under_control_1}
    For all $m,n \in \NN$ with $m<n$,
    \begin{align*}
        \widetilde{D}_h^{\xi}(\partial B_{e^{-n-50}}(z) , \partial B_{e^{-m}}(z)) \leq C \sum_{k=m}^{n-1} \min_{t \in [k,k+1]} \left(\exp\left((\xi-\delta) (h_{e^{-t}}(z)-2t)\right)\right).
    \end{align*}
    \item \label{it:around_under_control_1}
    For all $n \in \NN$,
    \begin{equation*}
        \widetilde{D}_h^{\xi}(\text{around}\,\,\mathbb{A}_{e^{-n-1},e^{-n}}(z)) \leq C \min_{t \in [n+1,n+2]} \left(\exp\left((\xi-\delta) (h_{e^{-t}}(z) - 2t)\right)\right).
    \end{equation*}
\end{enumerate}  
\end{lemma}

\begin{proof}
Fix $\alpha \in (0,1/4)$ such that $\alpha (1+\delta / (1+\xi_c)) < 1/4$. Since $\xi \mapsto Q(\xi)$ is continuous in $\xi$, we obtain by applying Lemma~\ref{lem:stringing_paths} with $A = 2(1+\xi_c) / \delta$ and by choosing the functions $b,M$ as in the statement of the lemma that there exists deterministic constant $C>1$ depending only on $\alpha,\varepsilon,\delta$, and $U$, such that the following holds for all $\xi \in (\xi_c/2,\xi_c]$. With probability at least $1-\varepsilon/100$, the following hold for all $z \in U$.
\begin{enumerate}
    \item \label{it:average_under_control_no_xi}
    \begin{equation*}
        h_{e^{-t}}(z) \leq 2t - \alpha \log t + C \quad \text{for all} \quad t >0.
    \end{equation*}
    \item \label{it:across_under_control_no_xi}
    For all $m,n \in \NN$ with $m<n$,
    \begin{align*}
        &\widetilde{D}_h^{\xi}(\partial B_{e^{-n-50}}(z) , \partial B_{e^{-m}}(z)) \\
        &\leq C \sum_{k=m}^{n-1} \min_{t \in [k,k+1]} \left(\exp\left(\xi h_{e^{-t}}(z) - \xi Q(\xi) t + \left(\frac{2(1+\xi)(2t - h_{e^{-t}}(z))}{A}\right)\right)\right).
    \end{align*}
    \item \label{it:around_under_control_no_xi}
    For all $n \in \NN$,
    \begin{align*}
      &\widetilde{D}_h^{\xi}(\text{around}\,\,\mathbb{A}_{e^{-n-1},e^{-n}}(z))\\
      &\leq C \min_{t \in [n+1,n+2]} \left(\exp\left(\xi h_{e^{-t}}(z) - \xi Q(\xi) t + \left(\frac{2(1+\xi)(2t-h_{e^{-t}}(z))}{A}\right)\right)\right).
    \end{align*}
\end{enumerate}
For the rest of the proof, we will assume that items~\eqref{it:average_under_control_no_xi}-\eqref{it:around_under_control_no_xi} hold.

Let $t_0 > 0$ be deterministic (depending only on $\alpha$ and $C$) such that $\alpha \log t - C>0$ for all $t \geq t_0$. Then, \eqref{it:average_under_control_no_xi} implies that
\begin{equation*}
2t-h_{e^{-t}}(z) \geq \alpha \log t - C >0 \quad \text{for all} \quad t \geq t_0, \,\,z \in U,
\end{equation*}
which implies that
\begin{align}\label{eq:large_t}
\left(\frac{2(1+\xi)}{A}\right)\left(2t - h_{e^{-t}}(z)\right) &= \left(\frac{\delta(1+\xi)}{1+\xi_c}\right)(2t-h_{e^{-t}}(z))\nonumber\\
&\leq \delta(2t-h_{e^{-t}}(z)) \quad \text{for all} \quad t \geq t_0,\,\,z \in U.
\end{align}
Moreover, since the process $(z,t) \mapsto h_{e^{-t}}(z)$ is continuous a.s., we obtain that there exists deterministic constant $C_1>0$ depending only on $\varepsilon,\delta,\alpha$, and $U$, such that with probability at least $1-\varepsilon/100$, we have that
\begin{equation}\label{eq:small_t}
\left(\frac{\delta (1+\xi)}{1+\xi_c)}\right)(2t - h_{e^{-t}}(z))  \leq \delta (2t - h_{e^{-t}}(z)) + C_1 \quad \text{for all} \quad 0<t \leq t_0,\,\,z \in U.   
\end{equation}
Suppose that \eqref{eq:small_t} holds as well and note that the probability that items~\eqref{it:average_under_control_no_xi}-\eqref{it:around_under_control_no_xi} and \eqref{eq:small_t} hold simultaneously is at least $1-\varepsilon/50$.

Combining \eqref{eq:large_t} with \eqref{eq:small_t} gives that
\begin{equation}\label{eq:all_t}
\left(\frac{\delta (1+\xi)}{1+\xi_c)}\right)(2t - h_{e^{-t}}(z)) \leq \delta (2t - h_{e^{-t}}(z)) + C_1 \quad \text{for all} \quad t \geq 0,\,\,z \in U.   
\end{equation}
It follows from combining item~\eqref{it:across_under_control_no_xi} with \eqref{eq:all_t} that for all $m,n \in \NN$ with $m<n$ and all $z \in U$,
\begin{align*}
    &\widetilde{D}_h^{\xi}(\partial B_{e^{-n-50}}(z),\partial B_{e^{-m}}(z))\\
    &\leq C \sum_{k=m}^{n-1} \min_{t \in [k,k+1]} \left(\exp\left(\xi h_{e^{-t}}(z) - \xi Q(\xi)t +\left(\frac{\delta(1+\xi)}{1+\xi_c}\right)(2t - h_{e^{-t}}(z))\right)\right)\\
    &\leq C \sum_{k=m}^{n-1} \min_{t \in [k,k+1]} \left(\exp\left(\xi h_{e^{-t}}(z) - \xi Q(\xi) t +\delta (2t - h_{e^{-t}}(z)) + C_1\right)\right)\\
    &\leq C e^{C_1} \sum_{k=m}^{n-1} \min_{t \in [k,k+1]} \left(\exp\left((\xi-\delta)(h_{e^{-t}}(z)-2t) + \xi (2-Q(\xi)) t\right)\right)\\
    &\leq C e^{C_1} \sum_{k=m}^{n-1} \min_{t \in [k,k+1]} \left(\exp\left((\xi-\delta)(h_{e^{-t}}(z)-2t)\right)\right).
\end{align*}
Similarly, combining item~\eqref{it:around_under_control_no_xi} with \eqref{eq:all_t} gives that for all $n \in \NN, z \in U$,
\begin{align*}
    \widetilde{D}_h^{\xi}(\text{around}\,\,\mathbb{A}_{e^{-n-1},e^{-n}}(z)) \leq C e^{C_1} \min_{t \in [n,n+1]} \left(\exp(\left(\xi-\delta)( h_{e^{-t}}(z)-2t)\right)\right).
\end{align*}
This completes the proof of the lemma.  
\end{proof}

\begin{lemma}\label{lem:around_poly}
Fix $\theta \in (0,\xi_c / 8)$. Then, there exists a choice of the functions $b,M$ (depending only on $\theta$) such that the following is true for all $\xi \in (\xi_c / 2, \xi_c]$. Let $U \subseteq \CC$ be a bounded open set and fix $\varepsilon \in (0,1)$. Then, there exists deterministic constant $C>0$ depending only on $\varepsilon,\theta$, and $U$, such that the following holds with probability at least $1-\varepsilon$. For all $z \in U$ and all $n \in \NN$, we have that
\begin{equation*}
    \widetilde{D}_h^{\xi}(\text{around} \,\, \mathbb{A}_{e^{-n-1},e^{-n}}(z)) \leq C n^{-\theta}.
\end{equation*}
\end{lemma}

\begin{proof}
We choose $\alpha \in (0,1/4)$ sufficiently close to $1/4$ and $\delta \in (0,\xi_c / 2)$ sufficiently close to $0$ (depending only on $\theta$) such that $(\xi_c /2 - \delta) \alpha \geq \theta$ and $\alpha (1+\delta / (1+\xi_c)) < 1/4$. We also choose the continuous functions $b,M$ such that $b(\xi) \sqrt{M(\xi)} > 40(1+\xi_c) / \delta$ for all $\xi>0$.

Applying Lemmas~\ref{lem:stringing_paths} and ~\ref{lem:stringing_paths_1} with $A = 2(1+\xi_c) / \delta$ gives that there exists a constant $C>1$ depending only on $\alpha,\delta,\varepsilon$, and $U$, such that for all $\xi \in (\xi_c / 2,\xi_c]$, we have with probability at least $1-\varepsilon$ that item~\eqref{it:average_under_control} in Lemma~\ref{lem:stringing_paths} and item~\eqref{it:around_under_control_1} in Lemma~\ref{lem:stringing_paths_1} both hold for the above choices of $\delta,\alpha$, and $C$. For the rest of the proof, we will assume that item~\eqref{it:average_under_control} in Lemma~\ref{lem:stringing_paths} and item~\eqref{it:around_under_control_1} in Lemma~\ref{lem:stringing_paths_1} both hold.

Since $(\xi-\delta) \alpha \geq (\xi_c / 2 - \delta)\alpha \geq \theta$ for all $\xi \in (\xi_c/2,\xi_c]$, we obtain that for all $n \in \NN, z \in U$,
\begin{align*}
    \widetilde{D}_h^{\xi}(\text{around}\,\, \mathbb{A}_{e^{-n-1},e^{-n}}(z))&\leq C \min_{t \in [n+1,n+2]} \left(\exp\left((\xi-\delta) (h_{e^{-t}}(z)-2t)\right)\right)\\
    &\leq C \min_{t \in [n+1,n+2]} \left(\exp\left((\xi-\delta)(-\alpha \log t + C)\right)\right)\\
    &\leq C \exp(C(\xi-\delta)) \exp(-(\xi-\delta) \alpha \log n)\\
    &\leq C \exp(C(\xi_c - \delta)) \exp(-(\xi_c/2 - \delta) \alpha \log n)\\
    &=C \exp(C(\xi_c-\delta)) n^{-(\xi_c / 2 - \delta)\alpha}\\
    &\leq C \exp(C(\xi_c - \delta)) n^{-\theta}.
\end{align*}
This completes the proof of the lemma.   
\end{proof}

\subsection{Bounding distances between scales for big and small circle averages for the GFF}
\label{subsec:distances_between_scales}

Now, we would like to use Lemma~\ref{lem:stringing_paths_1} to bound from above distances between different scales. In particular, we would like to bound the sum appearing on the right side of item~\eqref{it:across_under_control_1} in Lemma~\ref{lem:stringing_paths_1}. This is the main goal of this subsection. To bound from above the right side of item~\eqref{it:across_under_control_1} in Lemma~\ref{lem:stringing_paths_1}, we treat separately the case that the circle average $h_{e^{-t}}(z)$ is not too close to $2t$ and the case that $h_{e^{-t}}(z)$ is close to $2t$. 

We note that the arguments used in this subsection are similar to the arguments in \cite[Section~3.4]{ding2024critical}. However, as in Sections~\ref{subsec:estimates_across_around_annuli} and ~\ref{subsec:stringing_paths}, we will have to prove that the estimates that we obtain are uniform in the parameter $\xi \in (0,\xi_c)$ provided $\xi$ is bounded away from $0$.

We start by treating the case that $h_{e^{-t}}(z)$ is not too close to $2t$ in Lemma~\ref{lem:distance_circle_small} below.

\begin{lemma}\label{lem:distance_circle_small}
Fix $\theta > 0$ and $\beta > (1+\theta) / \xi_c$. Fix also $\xi_0 \in (((1+\theta)/\beta) \vee(\xi_c/2) , \xi_c)$. Then, there exists a choice of the functions $b,M$ (depending only on $\theta,\beta$, and $\xi_0$) such that the following holds for all $\xi \in [\xi_0,\xi_c]$. Fix $\varepsilon \in (0,1)$ and let $U \subseteq \CC$ be a bounded open set. Then, there exists deterministic constant $C>0$ (depending only on $\theta,\beta,\varepsilon,\xi_0$, and $U$) such that the following holds with probability at least $1-\varepsilon$. For all $z\in U$ and all $m,n \in \NN$ with $m<n$ such that
\begin{equation*}
    h_{e^{-t}}(z) - h_1(z) \leq 2t -\beta \log t \quad \text{for all} \quad t \in [m,n]_{\ZZ},
\end{equation*}
it holds that
\begin{equation*}
    \widetilde{D}_h^{\xi}(\partial B_{e^{-n-50}}(z),\partial B_{e^{-m}}(z)) \leq C m^{-\theta}.
\end{equation*}
\end{lemma}

\begin{proof}
Fix $\delta \in (0,\xi_c / 2)$ such that $(\xi_0 - \delta) \beta > 1+\theta$ in a way that depends only on $\xi_0,\beta$, and $\theta$. We also choose the functions $b,M$ such that $b(\xi) \sqrt{M(\xi)} > 40(1+\xi_c) / \delta$ for all $\xi>0$.

Fix $\xi \in [\xi_0 , \xi_c]$. Then, Lemma~\ref{lem:stringing_paths_1} implies that there exists deterministic constant $C_0>1$ depending only on $\varepsilon,\delta$, and $U$, such that with probability at least $1-\varepsilon/100$, we have that item~\eqref{it:across_under_control_1} in Lemma~\ref{lem:stringing_paths_1} occurs with $C_0$ in place of $C$. For the rest of the proof, we will assume that we are working on the event that item~\eqref{it:across_under_control_1} in Lemma~\ref{lem:stringing_paths_1} occurs.

Fix $z \in U$ and $m,n \in \NN$ with $m<n$, and such that
\begin{equation*}
    h_{e^{-t}}(z) - h_1(z) \leq 2t -\beta \log t \quad \text{for all} \quad t \in [m,n]_{\ZZ}.
\end{equation*}
Then, combining with item~\eqref{it:across_under_control_1} in Lemma~\ref{lem:stringing_paths_1}, we obtain that for all $m,n \in \NN, z \in U$ with $m<n$,
\begin{align}\label{eq:nice_bound}
\widetilde{D}_h^{\xi}(\partial B_{e^{-n-50}}(z) , \partial B_{e^{-m}}(z))&\leq C_0 \sum_{k=m}^{n-1} \min_{t \in [k,k+1]} \left(\exp\left((\xi-\delta) (h_1(z) - \beta \log t)\right)\right)\nonumber\\
&\leq C_0 \exp\left((\xi-\delta) \sup_{w \in U} |h_1(w)| \right) \sum_{k=m}^{n-1} k^{-(\xi-\delta)\beta}.
\end{align}
By the continuity of the circle average process of $h$, we get that there exists constant $C_1>0$ depending only on $\varepsilon$ and $U$ such that
\begin{equation}\label{eq:useful_bound}
\sup_{w \in U} |h_1(w)| \leq C_1 
\end{equation}
with probability at least $1-\varepsilon/100$. We assume that \eqref{eq:useful_bound} holds as well and note that the probability that both item~\eqref{it:across_under_control_1} in Lemma~\ref{lem:stringing_paths_1} and \eqref{eq:useful_bound} occur is at least $1-\varepsilon/50$.

Note that $(\xi-\delta) \beta \geq (\xi_0 - \delta) \beta > 1+\theta$, and so \eqref{eq:nice_bound} and \eqref{eq:useful_bound} together imply that
\begin{align}\label{eq:nice_bound_2}
\widetilde{D}_h^{\xi}(\partial B_{e^{-n-50}}(z) , \partial B_{e^{-m}}(z))&\leq C_0 \exp(C_1 \xi_c) \sum_{k=m}^{n-1} k^{-(\xi_0-\delta)\beta}.   
\end{align}
Note also that there exists a constant $C_2>0$ depending only on $\xi_0,\delta$, and $\beta$, such that
\begin{equation}\label{eq:useful_bound_2}
\sum_{k=m}^{\infty} k^{-(\xi_0 - \delta) \beta} \leq C_2 m^{-(\xi_0 - \delta)\beta + 1} \leq C_2 m^{-\theta}.
\end{equation}
Therefore, the proof of the lemma is complete by combining \eqref{eq:nice_bound_2} with \eqref{eq:useful_bound_2}.    
\end{proof}

Note that $3/4 < 1/\xi_c$. Hence, \cite[Proposition~2.5]{ding2024critical} implies that if $\beta' > (1+\theta) / \xi_c$, then it is a.s. the case that there exist points $z \in U$ and arbitrarily large $t>0$ such that $h_{e^{-t}}(z) - h_1(z) \geq 2t -\beta' \log t$. Therefore, we need to bound from above the sum on the right side of item~\eqref{it:across_under_control_1} in Lemma~\ref{lem:stringing_paths_1} in the case that $h_{e^{-t}}(z) - h_1(z) \geq 2t -\beta' \log t$ for some $\beta' > (1+\theta)/\xi_c$ fixed and deterministic. This is the purpose of Lemma~\ref{lem:distance_circle_large} below.

\begin{lemma}\label{lem:distance_circle_large}
Fix $\theta \in (0,\xi_c / 8)$ and $\beta > (1+\theta) / \xi_c$. Fix also $\xi_0 \in (((1+\theta)/\beta) \vee (\xi_c / 2) , \xi_c)$. Then, there exists a choice of the functions $b,M$ (depending only on $\theta,\beta$, and $\xi_0$) such that the following holds for all $\xi \in [\xi_0,\xi_c]$. Fix $\varepsilon \in (0,1)$ and let $U \subseteq \CC$ be a bounded open set. Then, there exists deterministic constant $C>0$ (depending only on $\theta,\beta,\varepsilon,\xi_0$, and $U$) such that the following holds with probability at least $1-\varepsilon$. For all $n \in \NN$ and all $z \in (e^{-n-100} \ZZ^2) \cap U$ such that
\begin{equation*}
    h_{e^{-n}}(z) - h_1(z) \geq 2n -\beta \log n,
\end{equation*}
we have that
\begin{equation*}
    \widetilde{D}_h^{\xi}(\partial B_{e^{-n-50}}(z) , \partial B_{e^{-n/2}}(z)) \leq C n^{-\theta}.
\end{equation*}  
\end{lemma}

As in the proof of \cite[Lemma~3.11]{ding2024critical} (which is the analog of Lemma~\ref{lem:distance_circle_large} in the context of \cite{ding2024critical}), the main inputs in the proof of Lemma~\ref{lem:distance_circle_large} are Lemma~\ref{lem:stringing_paths_1} and Lemma~\ref{lem:brownian_bridge} below.

\begin{lemma}(\cite[Lemma~3.12]{ding2024critical})
\label{lem:brownian_bridge}
Fix $0 < \alpha < \beta$. Let $T>0$, let $x \in [\alpha \log T, \beta \log T]$, and let $W$ be a Brownian bridge from $0$ to $2T - x$ in time $T$. For $S>0$,
\begin{equation*}
    \mathbb{P}\left[W_t \leq 2t -\alpha \log T,\,\text{for all}\,\,t \in [T/2,T],\,\,\int_{T/2}^T \mathds{1}_{\{W_t \geq 2t - \beta \log T\}} dt >S\right] \leq c_0 \exp\left(-c_1 \frac{S}{(\log T)^2}\right)
\end{equation*}
for constants $c_0,c_1>0$ depending only on $\alpha,\beta$.    
\end{lemma}

Lemma~\ref{lem:brownian_bridge} is applied for $W_t = h_{e^{-t}}(z) - h_1(z)$, $t \in [0,n]$, conditioned on $h_{e^{-n}}(z) - h_1(z)$ (see Steps 2 and 3 in the proof of \cite[Lemma~3.11]{ding2024critical} for more details).

We now give the proof of Lemma~\ref{lem:distance_circle_large}.

\begin{proof}[Proof of Lemma~\ref{lem:distance_circle_large}.]
First, we note that since $\beta \xi_0 > 1+\theta$ and $\theta < \xi_c / 8$, we can choose $\delta>0$ sufficiently small (in a way that depends only on $\beta,\xi_0$, and $\theta$) such that $\beta (\xi_0 - \delta) > 1+\theta$ and $4\theta < \xi_c / 2 -\delta$. Also, we can choose $\alpha \in (0,1/4)$ sufficiently close to $1/4$ (in a way that depends only on $\theta$ and $\delta$) such that $\alpha (\xi_c / 2 - \delta) > \theta$. For the rest of the proof, we will assume that we are working with the above choice of $\alpha$ and $\delta$.

Next, we fix $\xi \in [\xi_0 , \xi_c]$ and for all $n \in \NN, t>0$, and all $z \in (e^{-n-100} \ZZ^2) \cap U$, we consider the event 
\begin{equation*}
    H_n(z,t) = \{h_{e^{-t}}(z) - h_1(z) \geq 2t -\beta \log n\}.
\end{equation*}
Then, by arguing in the exact same way as in Steps 2 and 3 in the proof of \cite[Lemma~3.11]{ding2024critical} (see in particular equation $(3.66)$ in the proof of \cite[Lemma~3.11]{ding2024critical}) and using Lemma~\ref{lem:brownian_bridge}, we obtain that there exists $n_\star \in \NN$ depending only on $\alpha,\beta,\varepsilon$, and $U$, such that the following holds with probability at least $1-\varepsilon/100$. For all $n \in \NN, n\geq n_\star$, and all $z \in (e^{-n-100} \ZZ^2) \cap U$ such that
\begin{equation}\label{eq:circle_big}
h_{e^{-n}}(z) - h_1(z) \geq 2n -\beta \log n,  
\end{equation}
we have that
\begin{equation}\label{eq:event}
h_{e^{-t}}(z) - h_1(z) \leq 2t -\alpha \log n\,\,\text{for all} \,\,\,t \in [n/2,n],\,\,\,\text{and}\,\, \int_{n/2}^n \mathds{1}_{H_n(z,t)} dt \leq (\log n)^4.
\end{equation}

Lemma~\ref{lem:stringing_paths_1} implies that if we choose the functions $b,M$ such that $b(\xi') \sqrt{M(\xi')} > 40(1+\xi_c) / \delta$ for all $\xi' \in [\xi_0,\xi_c]$, we have that there exists a constant $C>0$ (depending only on $\varepsilon,\delta$, and $U$) such that the following is true with probability at least $1-\varepsilon/100$. For all $n \in \NN$ and all $z \in U$, it holds that
\begin{equation}\label{eq:upper_bound}
\widetilde{D}_h^{\xi}(\partial B_{e^{-n-50}}(z),\partial B_{e^{-n/2}}(z)) \leq C \int_{\lfloor n/2 \rfloor}^n \exp((\xi-\delta) (h_{e^{-t}}(z) - 2t)) dt.  
\end{equation}

Therefore, combining \eqref{eq:event} with \eqref{eq:upper_bound}, we get that the following is true with probability at least $1-\varepsilon/50$. For all $n \in \NN, n \geq n_\star$, and all $z \in (e^{-n-100}\ZZ^2) \cap U$ such that \eqref{eq:circle_big} occurs, we have that
\begin{align}\label{eq:many_integrals}
\widetilde{D}_h^{\xi}(\partial B_{e^{-n-50}}(z) , \partial B_{e^{-n/2}}(z))&\leq C \int_{\lfloor n/2 \rfloor}^n \exp((\xi-\delta)(h_{e^{-t}}(z)-2t))\mathds{1}_{H_n(z,t)}dt\nonumber\\
&+C \int_{\lfloor n/2 \rfloor}^n \exp((\xi-\delta)(h_{e^{-t}}(z)-2t))\mathds{1}_{H_n(z,t)^c} dt\nonumber\\
&\leq C \int_{\lfloor n/2 \rfloor}^n \exp((\xi-\delta)(h_1(z) - \alpha \log n)) \mathds{1}_{H_n(z,t)} dt\nonumber\\
&+C \int_{\lfloor n/2 \rfloor}^n \exp((\xi-\delta)(h_1(z)-\beta \log n)) \mathds{1}_{H_n(z,t)^c} dt\nonumber\\
&\leq C \exp\left((\xi_c-\delta) \sup_{w \in U} |h_1(w)| \right) \left(n^{-\alpha (\xi-\delta)} \int_{\lfloor n/2 \rfloor}^n \mathds{1}_{H_n(z,t)} dt + n^{-\beta(\xi-\delta)+1}\right)\nonumber\\
&\leq C \exp\left((\xi_c - \delta) \sup_{w \in U} |h_1(w)|\right) \left(n^{-\alpha (\xi-\delta)} (\log n)^4 + n^{-\beta(\xi-\delta)+1}\right).
\end{align}

Note that the choices of $\alpha$ and $\delta$ imply that $\alpha(\xi-\delta) > \theta$ and $\beta(\xi-\delta) -1> \theta$. Moreover, we have that $\sup_{w \in U}|h_1(w)|$ is finite a.s., and so by possibly taking $C$ to be larger (in a way that depends only on $\varepsilon,\delta,\alpha,\beta$, and $U$) and combining with \eqref{eq:many_integrals}, we obtain that with probability at least $1-\varepsilon/25$, it holds that
\begin{equation}\label{eq:great_bound}
\widetilde{D}_h^{\xi}(\partial B_{e^{-n-50}}(z),\partial B_{e^{-n/2}}(z)) \leq C \left(n^{-\theta} (\log n)^4 + n^{-\theta}\right)
\end{equation}
for all $n \in \NN, n \geq n_\star$, and all $z \in (e^{-n-100} \ZZ^2) \cap U$ for which \eqref{eq:circle_big} holds.

Finally, Proposition~\ref{prop:main_lqg_estimate} implies that there exists $C_1>0$ (depending only on $\varepsilon,\theta,U$, and $n_\star$) such that with probability at least $1-\varepsilon/100$, it holds that
\begin{equation}\label{eq:great_bound_small_n}
\widetilde{D}_h^{\xi}(\partial B_{e^{-n-50}}(z),\partial B_{e^{-n/2}}(z)) \leq C_1 n^{-\theta}\,\,\text{for all}\,\, 1 \leq n \leq n_\star,\,\,z \in (e^{-n-100} \ZZ^2) \cap U.
\end{equation}
Therefore, the proof of the lemma is complete by combining \eqref{eq:great_bound} with \eqref{eq:great_bound_small_n}. 
\end{proof}

Next, we would like to have a version of Lemma~\ref{lem:distance_circle_large} which can be applied for \emph{any} point $z \in U$. This is achieved in Lemma~\ref{lem:large_at_some_t} below.

\begin{lemma}\label{lem:large_at_some_t}
Fix $\theta \in (0,\xi_c / 8)$ and $\beta > (1+\theta) / \xi_c$. Fix also $\xi_0 \in (((1+\theta) / \beta) \vee (\xi_c / 2) , \xi_c)$. Then, there exists a choice of the functions $b,M$ (depending only on $\theta,\beta$, and $\xi_0$) such that the following holds for all $\xi \in [\xi_0,\xi_c]$. Fix $\varepsilon \in (0,1)$ and let $U \subseteq \CC$ be a bounded open set. Then, there exists a deterministic constant $C>0$ depending only on $\theta,\varepsilon,\beta,\xi_0$, and $U$, such that the following is true with probability at least $1-\varepsilon$. For all $n \in \NN$ and all $z \in U$ for which there exists $t \in [n+1,n+2]$ such that
\begin{equation}\label{eq:large_for_some_t}
h_{e^{-t}}(z) - h_1(z) \geq 2t -\beta \log t,   
\end{equation}
we have that
\begin{equation*}
    \widetilde{D}_h^{\xi}(\partial B_{e^{-n-40}}(z) , \partial B_{e^{-n/2 - 1}}(z)) \leq C n^{-\theta}.
\end{equation*}   
\end{lemma}

\begin{proof}
Fix $\xi \in [\xi_0,\xi_c]$ and $\beta' > \beta$. Then, Lemma~\ref{lem:distance_circle_large} implies that there exists a choice of the functions $b,M$ (depending only on $\theta,\beta$, and $\xi_0$) such that the following holds. There exists a constant $C_1>0$ (depending only on $\theta,\beta',\varepsilon,\xi_0$, and $U$) such that the following is true with probability at least $1-\varepsilon/100$. For all $n \in \NN$ and all $z \in (e^{-n-100} \ZZ^2) \cap B_1(U)$ such that 
\begin{equation}\label{eq:average_large}
    h_{e^{-n}}(z) - h_1(z) \geq 2n -\beta' \log n,
\end{equation}
it holds that
\begin{equation}\label{eq:good_bound_1}
\widetilde{D}_h^{\xi}(\partial B_{e^{-n-50}}(z) , \partial B_{e^{-n/2}}(z)) \leq C_1 n^{-\theta}.
\end{equation}
Moreover, item~\eqref{it:average_under_control} in Lemma~\ref{lem:stringing_paths} implies that there exists a deterministic constant $t_\star>0$ (depending only on $\varepsilon$ and $U$) such that with probability at least $1-\varepsilon/100$, we have that
\begin{equation}\label{eq:good_bound_2}
    h_{e^{-t}}(z) \leq 2t \quad \text{for all} \quad t \geq t_\star,\,\, z \in B_1(U).
\end{equation}

Let $A>1$ be sufficiently large (depending only on $\beta$ and $\beta'$) such that $\beta (1+2/A) < \beta'$. Then, Lemma~\ref{lem:uniform_circle_ave} implies that there exists a constant $C_2>0$ depending only on $\varepsilon,A$ and $U$, such that with probability at least $1-\varepsilon / 100$, the condition in the statement of Lemma~\ref{lem:uniform_circle_ave} holds with $C_2$ in place of $C_1$. Furthermore, there exists a constant $C_3>0$ depending only on $\varepsilon$ and $U$ such that with probability at least $1-\varepsilon/100$,
\begin{equation}\label{eq:good_bound_3}
\sup_{w \in B_1(U)} |h_1(w)| \leq C_3.   
\end{equation}
For the rest of the proof, we will assume that we are working on the event that \eqref{eq:good_bound_1},\eqref{eq:good_bound_2},\eqref{eq:good_bound_3}, and the condition in the statement of Lemma~\ref{lem:uniform_circle_ave} hold. Note that the above event occurs with probability at least $1-\varepsilon/25$.

Now, fix $z \in U$ and $t \in [n+1,n+2]$ such that \eqref{eq:large_for_some_t} holds. Let also $z_n \in (e^{-n-100} \ZZ^2) \cap B_1(U)$ be such that $|z-z_n| \leq e^{-n-99}$. Then, Lemma~\ref{lem:uniform_circle_ave} implies that 
\begin{equation*}
    h_{e^{-n}}(z_n) - h_1(z_n) \geq h_{e^{-t}}(z) - \left(\frac{2(2t-h_{e^{-t}}(z))}{A}\right) -|h_1(z_n)| - C_2.
\end{equation*}
Moreover, by \eqref{eq:good_bound_2} and if we further assume that $t \geq t_\star$, we obtain that
\begin{equation*}
|2t - h_{e^{-t}}(z)| = 2t - h_{e^{-t}}(z) \leq \beta \log t + |h_1(z)|.    
\end{equation*}
It follows that
\begin{align*}
h_{e^{-n}}(z_n) - h_1(z_n) &\geq 2t - \beta \log t -|h_1(z)| - \left(\frac{2(\beta \log t +|h_1(z)|)}{A}\right) -|h_1(z_n)| - C_2\\
&=2t - \beta \left(1+\frac{2}{A}\right) \log t -\left(1+\frac{2}{A}\right) |h_1(z)| - |h_1(z_n)| - C_2\\
&\geq 2n -\beta \left(1+\frac{2}{A}\right) \log n -C_3 \left(2+\frac{2}{A}\right) - C_2 -C_4,  
\end{align*}
where $C_4>0$ depends only on $\beta$ and $A$. Therefore, there exists $n_\star \in \NN, n_\star \geq t_\star$ depending only on $\beta,t_\star,C_2,C_3,C_4,A$, and $\beta'$, such that
\begin{equation*}
    h_{e^{-n}}(z_n) - h_1(z_n) \geq 2n -\beta' \log n.
\end{equation*}
Thus, it follows from combining with \eqref{eq:good_bound_1} that 
\begin{equation*}
\widetilde{D}_h^{\xi}(\partial B_{e^{-n-50}}(z_n) , \partial B_{e^{-n/2}}(z_n)) \leq C_1 n^{-\theta}  
\end{equation*}
for all $n \in \NN, n \geq n_\star$, and all $z \in (e^{-n-100} \ZZ^2) \cap B_1(U)$ for which \eqref{eq:average_large} holds. Since $|z-z_n| \leq e^{-n-99}$, we have that
\begin{equation*}
\mathbb{A}_{e^{-n-40},e^{-n/2-1}}(z) \subseteq \mathbb{A}_{e^{-n-50},e^{-n/2}}(z_n)
\end{equation*}
and so
\begin{equation}\label{eq:good_bound_4}
\widetilde{D}_h^{\xi}(\partial B_{e^{-n-40}}(z) ,\partial B_{e^{-n/2-1}}(z)) \leq \widetilde{D}_h^{\xi}(\partial B_{e^{-n-50}}(z_n) , \partial B_{e^{-n/2}}(z_n)) \leq C_1 n^{-\theta}. 
\end{equation}

Finally, Proposition~\ref{prop:main_lqg_estimate} implies that there exists a constant $\widetilde{C}>0$ depending only on $\varepsilon,\theta$, and $U$ such that with probability at least $1-\varepsilon/100$, we have that
\begin{equation}\label{eq:good_bound_5}
\widetilde{D}_h^{\xi}(\partial B_{e^{-n-50}}(w) , \partial B_{e^{-n/2}}(w)) \leq \widetilde{C} n^{-\theta}\,\,\text{for all}\,\, 1\leq n \leq n_\star,\,\,w \in (e^{-n-100} \ZZ^2) \cap B_1(U).   
\end{equation}
Therefore, the proof of the lemma is complete by combining \eqref{eq:good_bound_4} with \eqref{eq:good_bound_5}.   
\end{proof}

\subsection{Proofs of Theorem~\ref{thm:main_result} and Corollary~\ref{cor:main_result}}
\label{subsec:proofs_of_main_results}

In this subsection, we will complete the proofs of Theorem~\ref{thm:main_result} and Corollary~\ref{cor:main_result}. We start by giving in Lemma~\ref{lem:distance_poly} below a useful estimate on $\widetilde{D}_h^{\xi}(\partial B_{e^{-n-50}}(z),\partial B_{e^{-n/2 - 1}}(z))$ which holds for all $n \in \NN$ and all $z \in U$. It will follow from Lemmas~\ref{lem:distance_circle_small} and ~\ref{lem:large_at_some_t}.

\begin{lemma}\label{lem:distance_poly}
Fix $\theta \in (0,\xi_c / 8)$. Then, there exists $\xi_0 \in (0,\xi_c)$ depending only on $\theta$ and a choice of the functions $b,M$ (depending only on $\theta$) such that the following is true for all $\xi \in [\xi_0,\xi_c]$. Fix $\varepsilon \in (0,1)$ and a bounded open set $U \subseteq \CC$. Then, there exists deterministic constant $C>0$ (depending only on $\varepsilon,\theta$, and $U$) such that with probability at least $1-\varepsilon$, we have that
\begin{equation*}
    \widetilde{D}_h^{\xi}(\partial B_{e^{-n-50}}(z) , \partial B_{e^{-n/2 - 1}}(z)) \leq C n^{-\theta}\,\,\text{for all}\,\, n \in \NN,z \in U.
\end{equation*}
\end{lemma}

\begin{proof}
Fix $\beta > (1+\theta) / \xi_c$ and $\xi_0 \in (((1+\theta) / \xi_c) \vee (\xi_c / 2) , \xi_c)$ (both depending only on $\theta$). Then, Lemma~\ref{lem:distance_circle_small} implies that there exists a choice of the functions $b,M$ (depending only on $\beta,\theta$, and $\xi_0$) such that the following holds for all $\xi \in [\xi_0,\xi_c]$. There exists a deterministic constant $C_1>0$ (depending only on $\theta,\varepsilon,\beta,\xi_0$, and $U$) such that with probability at least $1-\varepsilon/100$, we have that
\begin{equation}\label{eq:amazing_bound_1}
\widetilde{D}_h^{\xi}(\partial B_{e^{-n-50}}(z) , \partial B_{e^{-n/2-1}}(z)) \leq C_1 n^{-\theta} \end{equation}
for all $n \in \NN$, and all $z \in U$ such that
\begin{equation*}
h_{e^{-t}}(z) - h_1(z) \leq 2t -\beta \log t \quad \text{for all} \quad t \in [n/2 , n]_{\ZZ}.   
\end{equation*}

Suppose that we are working on the event that \eqref{eq:amazing_bound_1} holds. To complete the proof of the lemma, it suffices to show that \eqref{eq:amazing_bound_1} holds in the case that $h_{e^{-t}}(z) \geq 2t -\beta \log t$ for some $t \in [n/2,n]_{\ZZ}$. Fix $z \in U$ and suppose that we are working on the latter event. Let $\tau$ denote the largest time $t \in [\lfloor n/2 \rfloor , n]_{\ZZ}$ for which $h_{e^{-t}}(z) - h_1(z) \geq 2t -\beta \log t$. Let also $N \in [\lfloor n/2 \rfloor - 1,n-2]_{\ZZ}$ be chosen such that $\tau \in [N+1,N+2]_{\ZZ}$. Note that by the definitions of $\tau$ and $N$, we have that $h_{e^{-t}}(z) - h_1(z) \geq 2t -\beta \log t$ for some $t \in [N+1,N+2]_{\ZZ}$. Thus, Lemma~\ref{lem:large_at_some_t} implies that by possibly modifying the choice of the functions $b,M$ (in a way that depends only on $\theta,\beta$, and $\xi_0$), we get that there exists a deterministic constant $C_2>0$ depending only on $\varepsilon,\theta,\beta,\xi_0$, and $U$, such that with probability at least $1-\varepsilon/100$, we have that
\begin{equation}\label{eq:amazing_bound_2}
\widetilde{D}_h^{\xi}(\partial B_{e^{-N-40}}(z) , \partial B_{e^{-N/2 - 1}}(z)) \leq C_2 N^{-\theta}.  
\end{equation}

Note that the definitions of $\tau$ and $N$ imply that $h_{e^{-t}}(z) - h_1(z) \leq 2t - \beta \log t$ for all $t \in [N+2,n]$. Therefore, Lemma~\ref{lem:distance_circle_small} implies that there exists a deterministic constant $C_3>0$ (depending only on $\theta,\beta,\varepsilon,\xi_0$, and $U$) such that
\begin{equation}\label{eq:amazing_bound_3}
\widetilde{D}_h^{\xi}(\partial B_{e^{-n-50}}(z) , \partial B_{e^{-N-2}}(z)) \leq C_3 N^{-\theta}   \end{equation}
with probability at least $1-\varepsilon/100$. Furthermore, Lemma~\ref{lem:around_poly} implies that there exists a constant $C_4>0$ depending only on $\varepsilon,\theta$, and $U$ such that
\begin{equation}\label{eq:amazing_bound_4}
\widetilde{D}_h^{\xi}(\text{around}\,\,\mathbb{A}_{e^{-N-21},e^{-N-20}}(z)) \leq C_4 N^{-\theta}   
\end{equation}
with probability at least $1-\varepsilon/100$. We note that we can choose the functions $b,M$ so that the statements of Lemmas~\ref{lem:around_poly}, ~\ref{lem:distance_circle_small}, and ~\ref{lem:large_at_some_t} hold simultaneously for the above choice of $b,M$.

Suppose that we are working on the event that \eqref{eq:amazing_bound_1},\eqref{eq:amazing_bound_2},\eqref{eq:amazing_bound_3}, and \eqref{eq:amazing_bound_4} all hold simultaneously and note that the probability of that event is at least $1-\varepsilon/25$. Then, we have by \eqref{eq:amazing_bound_2} that
\begin{equation}\label{eq:amazing_bound_5}
\widetilde{D}_h^{\xi}(\partial B_{e^{-N-40}}(z) , \partial B_{e^{-n/2-1}}(z)) \leq C_2 N^{-\theta}.
\end{equation}

Finally, we note that the union of any path from $\partial B_{e^{-N-40}}(z)$ to $\partial B_{e^{-n/2-1}}(z)$, any path from $\partial B_{e^{-n-50}}(z)$ to $\partial B_{e^{-N-2}}(z)$, and any path in $\mathbb{A}_{e^{-N-21},e^{-N-20}}(z)$ disconnecting the inner and outer boundaries of the latter annulus is connected and contains a path from $\partial B_{e^{-n-50}}(z)$ to $\partial B_{e^{-n/2 - 1}}$. Therefore, the proof of the lemma is complete by combining \eqref{eq:amazing_bound_3}, \eqref{eq:amazing_bound_4}, and \eqref{eq:amazing_bound_5}.    
\end{proof}

It will also be useful to obtain a version of Lemma~\ref{lem:distance_poly} at dyadic scales. This is the purpose of Lemma~\ref{lem:distance_dyadic} below.

\begin{lemma}\label{lem:distance_dyadic}
Fix $\theta \in (0,\xi_c / 8)$. Then, there exists $\xi_0 \in (0,\xi_c)$ and a choice of the functions $b,M$ (both depending only on $\theta$) such that the following is true for all $\xi \in [\xi_0 , \xi_c]$. Fix $\varepsilon \in (0,1)$ and a bounded open set $U \subseteq \CC$. Then, there exists deterministic $C>0$ (depending only on $\varepsilon,\theta$, and $U$) such that with probability at least $1-\varepsilon$, we have that
\begin{equation*}
    \widetilde{D}_h^{\xi}(\partial B_{2^{-2^k}}(z) , \partial B_{2^{-2^{\ell}}}(z)) \leq C 2^{-\theta \ell} \quad \text{for all}\quad z \in U,\,\,\ell,k \in \NN,\,\,\ell <k.
\end{equation*}   
\end{lemma}

\begin{proof}
Let $\xi_0 \in (0,\xi_c)$ denote the constant in the statement of Lemma~\ref{lem:distance_poly} and note that the proof of Lemma~\ref{lem:distance_poly} implies that we can choose $\xi_0$ such that $\xi_0 > \xi_c/2$. 

Fix $\xi \in [\xi_0,\xi_c]$. Then, combining Lemmas~\ref{lem:around_poly} and ~\ref{lem:distance_poly}, we obtain that there exists a choice of the functions $b,M$ (depending only on $\theta$ and $\xi_0$) such that with probability at least $1-\varepsilon$, we have that
\begin{equation}\label{eq:across_dyadic}
\widetilde{D}_h^{\xi}(\partial B_{e^{-2^j - 50}}(z) , \partial B_{e^{-2^{j-1}-1}}(z)) \leq C 2^{-\theta j} \quad \text{for all}\quad j \in \NN,\,\,z \in U,    
\end{equation}
and 
\begin{equation}\label{eq:around_dyadic}
\widetilde{D}_h^{\xi}(\text{around}\,\, \mathbb{A}_{e^{-2^j -2},e^{-2^j-1}}(z)) \leq C 2^{-\theta j} \quad \text{for al} \quad j \in \NN,\,\,z \in U.  
\end{equation}
For the remaining of the proof, we will assume that we are working on the event that both \eqref{eq:across_dyadic} and \eqref{eq:around_dyadic} hold.

Note that for all $j \in \NN$, the union of any path from $\partial B_{e^{-2^{j+1}-50}}(z)$ to $\partial B_{e^{-2^j - 1}}(z)$, any path from $\partial B_{e^{-2^j -50}}(z)$ to $\partial B_{e^{-2^{j-1}-1}}(z)$, and any path in $\mathbb{A}_{e^{-2^j -2},e^{-2^j-1}}(z)$ which disconnects the inner from the outer boundaries of $\mathbb{A}_{e^{-2^j -2},e^{-2^j-1}}(z)$ is connected. Therefore, if we consider paths which attain the minimal $\widetilde{D}_h^{\xi}$-distances in \eqref{eq:across_dyadic} and \eqref{eq:around_dyadic} for $j=\ell-1,\cdots,k$, then the union of these paths contains a path from $\partial B_{e^{-2^k}}(z)$ to $\partial B_{e^{-2^{\ell}}}(z)$ whose $\widetilde{D}_h^{\xi}$-length is at most 
\begin{equation*}
    2C \sum_{j=\ell-1}^k 2^{-\theta j} \leq 2C \sum_{j=\ell-1}^{\infty} 2^{-\theta j} = \left(\frac{2C}{1-2^{-\theta}}\right) 2^{-\theta (\ell-1)}.
\end{equation*}
This completes the proof of the lemma.   
\end{proof}

Next, we combine Lemmas~\ref{lem:around_poly} and ~\ref{lem:distance_dyadic} to prove that the following stronger version of Proposition~\ref{prop:quantitative_estimate} holds.

\begin{proposition}\label{prop:log_modulus}
Fix $\theta \in (0,\xi_c / 8)$. Then, there exists $\xi_0 \in (0,\xi_c)$ and a choice of the functions $b,M$ (depending only on $\theta$) such that the following holds for all $\xi \in [\xi_0,\xi_c]$. Fix $\varepsilon \in (0,1)$ and let $U \subseteq \CC$ be an open and connected set. Let also $K \subseteq U$ be a compact and connected set with more than one point. Then, there exists a deterministic constant $C>0$ (depending only on $\varepsilon,\theta,K$, and $U$) such that
\begin{equation*}
    \widetilde{D}_h^{\xi}(z,w ; U) \leq C \left(\max\left\{1,\log\left(\frac{1}{|z-w|}\right)\right\}\right)^{-\theta}\quad \text{for all} \quad z,w \in K
\end{equation*}
with probability at least $1-\varepsilon$.
\end{proposition}

\begin{proof}
Let $\xi_0 \in (\xi_c / 2,\xi_c)$ and the choice of the functions $b,M$ be such that the statements of Lemmas~\ref{lem:around_poly} and ~\ref{lem:distance_dyadic} hold and fix $\xi \in [\xi_0,\xi_c]$. Let also $V \subseteq U$ be a bounded open and connected set such that $\overline{V} \subseteq U$ and $K \subseteq V$. Fix also
\begin{equation*}
    r \in \left(0,\left(\frac{\text{dist}(V , \partial U)}{100}\right)^4 \wedge e^{-100}\right).
\end{equation*}
Then, combining Lemmas~\ref{lem:around_poly} and ~\ref{lem:distance_dyadic} gives that there exists a constant $C>0$ (depending only on $\varepsilon,\theta$, and $V$) such that with probability at least $1-\varepsilon$, we have that
\begin{equation}\label{eq:around_dyadic_1}
\widetilde{D}_h^{\xi}(\text{around}\,\, \mathbb{A}_{e^{-2^{\ell}-2},e^{-2^{\ell}-1}}(z)) \leq C 2^{-\theta \ell} \quad \text{for all} \quad \ell \in \NN,\,\,z \in V,   
\end{equation}
and
\begin{equation}\label{eq:across_dyadic_1}
\widetilde{D}_h^{\xi}(\partial B_{e^{-2^k}}(z) , \partial B_{e^{-2^{\ell}}}(z)) \leq C 2^{-\theta \ell} \quad \text{for all} \quad z \in V,\,\,k,\ell \in \NN\,\,\text{with}\,\,\ell < k.
\end{equation}
Note that sending $k \to \infty$ in \eqref{eq:across_dyadic_1} gives that
\begin{equation}\label{eq:across_dyadic_2}
\widetilde{D}_h^{\xi}(z,\partial B_{e^{-2^{\ell}}}(z)) \leq C 2^{-\theta \ell} \quad \text{for all} \quad \ell \in \NN,\,\, z \in V.   
\end{equation}

Suppose that we are working on the event that both \eqref{eq:around_dyadic_1} and \eqref{eq:across_dyadic_1} hold. Let $x_1,\cdots,x_N \in \CC$ be such that $V \subseteq \cup_{j=1}^N B_{r/2}(x_j)$ and $V \cap B_{r/2}(x_j) \neq \emptyset$ for all $j=1,\cdots,N$. Let $z,w \in V$ be such that $|z-w| \leq r$. Let also $\ell=\ell(z,w) \in \NN$ be such that $|z-w| \in [e^{-2^{\ell+2}},e^{-2^{\ell+1}})$. 

Note that $w \in B_{e^{-2^{\ell}-2}}(z)$ and $B_{e^{-2^{\ell}-1}}(z) \subseteq B_{e^{-2^{\ell}}}(w)$ since $|z-w| < e^{-2^{\ell+1}} \leq e^{-2^{\ell}-2}$. Therefore, the union of any path from $w$ to $\partial B_{e^{-2^{\ell}}}(w)$, any path from $z$ to $\partial B_{e^{-2^{\ell}}}(z)$, and any path in $\mathbb{A}_{e^{-2^{\ell}-2},e^{-2^{\ell}-1}}(z)$ which disconnects the inner and outer boundaries of the latter annulus is connected. Note also that since $|z-w| \leq r < (\text{dist}(V , \partial U))^4$, the choice of $\ell$ implies that 
\begin{equation*}
e^{-2^{\ell}} \leq |z-w|^{1/4} < \text{dist}(V , \partial U).
\end{equation*}
In particular, any path from $w$ to $\partial B_{e^{-2^{\ell}}}(w)$, any path from $z$ to $\partial B_{e^{-2^{\ell}}}(z)$, and any path in $\mathbb{A}_{e^{-2^{\ell}-2},e^{-2^{\ell}-1}}(z)$ disconecting the inner from the outer boundaries of the latter annulus, have to be all contained in $U$.

Hence, we obtain by combining \eqref{eq:around_dyadic_1} with \eqref{eq:across_dyadic_2} that 
\begin{equation}\label{eq:log_modulus_2}
 \widetilde{D}_h^{\xi}(z,w ; U) \leq 3C 2^{-\theta \ell} \leq \widetilde{C} (\log(1/|z-w|))^{-\theta},  
\end{equation}
where $\widetilde{C}>0$ is a deterministic constant that depends only on $C$ and $\theta$.

Finally, to complete the proof of the proposition, we treat the case that $z,w \in K$ and $|z-w| > r$. Let $P$ be a path in $V$ connecting $z$ to $w$ and let $B_1,\cdots,B_N$ denote the Euclidean balls in $\{B_{r/2}(x_j)\}_{j=1}^N$ that $P$ intersects in chronological order. Note that $z \in B_1$ and for all $j \in \{2,\cdots,n\}$, there exists $z_j \in B_{j-1} \cap B_j$. Let $j_0 \in \{2,\cdots,n\}$ be such that $w \in B_{j_0}$. Then, we have that
\begin{align*}
\widetilde{D}_h^{\xi}(z,w) &\leq \widetilde{D}_h^{\xi}(z,z_1) + \widetilde{D}_h^{\xi}(z_{j_0},w) + \sum_{j=1}^{j_0-1} \widetilde{D}_h^{\xi}(z_j,z_{j+1})\\
&\leq \widetilde{C} \left(\left(\log\left(\frac{1}{|z-z_1|}\right)\right)^{-\theta} + \left(\log\left(\frac{1}{|z_{j_0}-w|}\right)\right)^{-\theta} + \sum_{j=1}^{j_0-1} \left(\log\left(\frac{1}{|z_j-z_{j+1}|}\right)\right)^{-\theta}\right)\\
&\leq \widetilde{C} N \left(\max\left\{1,\log\left(\frac{1}{|z-w|}\right)\right\}\right)^{-\theta}.   
\end{align*}
This completes the proof of the proposition.  
\end{proof}

Next, we fix a sequence $(\xi_n)_{n \in \NN}$ in $(0,\xi_c)$ such that $\xi_n \to \xi_c$ as $n \to \infty$. Using Proposition~\ref{prop:log_modulus}, we will show in Proposition~\ref{prop:tightness} below that for any open and connected set $U \subseteq \CC$, the collection of the laws of the random metrics $\{\widetilde{D}_h^{\xi_n}(\cdot , \cdot ; U)\}_{n \in \NN}$ is tight with respect to the local uniform topology on $U \times U$.

\begin{proposition}\label{prop:tightness}
There is a choice of the functions $b,M$ (which is independent of the sequence $(\xi_n)_{n \in \NN}$) such that the following holds. Let $U \subseteq \CC$ be an open and connected set (including $\CC$). Then, the collection of laws of the random internal metrics $\{\widetilde{D}_h^{\xi_n}(\cdot,\cdot ; U)\}_{n \in \NN}$ on $U \times U$ is tight with respect to the local uniform topology on $U \times U$.   
\end{proposition}

\begin{proof}
Let $\xi_0 \in (0,\xi_c)$ be the constant in Proposition~\ref{prop:log_modulus} which corresponds to $\theta = \xi_c / 16$. Then, we choose the functions $b,M$ as in the statement of Proposition~\ref{prop:log_modulus} corresponding to the above choice of $\theta$.

Fix $\varepsilon \in (0,1)$ and let $(K_n)_{n \in \NN}$ be a sequence of compact and connected subsets of $U$ such that $K_m \subseteq K_{m+1}$ for all $m \in \NN$ and $U = \cup_{m \in \NN} K_m$. Then, Proposition~\ref{prop:log_modulus} implies that for all $n \in \NN$, there exists a constant $R_n >0$ depending only on $\varepsilon,n,K_n$, and $U$ such that
\begin{equation}\label{eq:tightness}
\mathbb{P}\left[\sup_{z,w \in K_n,z\neq w} \frac{\widetilde{D}_h^{\xi_m}(z,w;U)}{\left(\max\left\{1,\log(1/|z-w|)\right\}\right)^{-\theta}}\geq R_n\right] \leq \varepsilon/2^n
\end{equation}
for all $m \in \NN$ such that $\xi_m \geq \xi_0$.

Let $\mathcal{K}$ denote the set of functions $f$ on $U \times U \rightarrow \RR$ such that
\begin{equation*}
|f(z,w) - f(z',w')| \leq 2R_n \left(\max\left\{1,\log\left(\frac{1}{|z-z'| + |w-w'|}\right)\right\}\right)^{-\theta}  
\end{equation*}
for all $n \in \NN$ and all $(z,w),(z',w') \in K_n \times K_n$.

Note that if
\begin{equation*}
\sup_{z,w \in K_n,z\neq w} \frac{\widetilde{D}_h^{\xi_m}(z,w;U)}{\left(\max\left\{1,\log(1/|z-w|)\right\}\right)^{-\theta}}\leq R_n    
\end{equation*}
for some $n,m \in \NN$, we have that
\begin{align*}
|\widetilde{D}_h^{\xi_m}(z,w;U) - \widetilde{D}_h^{\xi_m}(z',w';U)|&\leq |\widetilde{D}_h^{\xi_m}(z,w;U) - \widetilde{D}^{\xi_m}(z',w;U)|+|\widetilde{D}_h^{\xi_m}(z',w;U) - \widetilde{D}_h^{\xi_m}(z',w';U)|\\  
&\leq \widetilde{D}_h^{\xi_m}(z,z';U) + \widetilde{D}_h^{\xi_m}(w,w';U)\\
&\leq R_n \left(\left(\max\left\{1,\log\left(\frac{1}{|z-z'|}\right)\right\}\right)^{-\theta} + \left(\max\left\{1,\log\left(\frac{1}{|w-w'|}\right)\right\}\right)^{-\theta}\right)\\
&\leq 2R_n \left(\max\left\{1,\log\left(\frac{1}{|z-z'| + |w-w'|}\right)\right\}\right)^{-\theta}.
\end{align*}

Thus, combining with \eqref{eq:tightness}, we obtain that
\begin{equation*}
  \mathbb{P}[\widetilde{D}_h^{\xi_m}(\cdot,\cdot;U) \in \mathcal{K}] \geq 1-\varepsilon \quad \text{for all} \quad m \in \NN \quad \text{with} \quad \xi_m \geq \xi_0.
\end{equation*}
Moreover, let $m_0 \in \NN$ be such that $\xi_m \geq \xi_0$ for all $m \geq m_0$. Then, \cite[Proposition~3.12]{kavvadias2025continuity} implies that for all $n \in \NN$, by possibly taking $R_n$ to be larger (in a way that depends only on $\varepsilon,n,K_n,U$, and the sequence $(\xi_m)_{m \in \NN}$), we have that \eqref{eq:tightness} holds for all $1 \leq m \leq m_0$ for that choice of $R_n$. In particular, we have that
\begin{equation*}
    \mathbb{P}[\widetilde{D}_h^{\xi_m}(\cdot,\cdot ; U) \in \mathcal{K}] \geq 1-\varepsilon \quad \text{for all} \quad m \in \NN.
\end{equation*}

Note that the Arz\'ela-Ascoli theorem implies that the set $\mathcal{K}$ is compact with respect to the local uniform topology of real-valued functions on $U \times U$. Therefore, we obtain that $\{\widetilde{D}_h^{\xi_n}(\cdot,\cdot;U)\}_{n \in \NN}$ is tight with respect to the local uniform topology of functions on $U \times U$, since $\varepsilon \in (0,1)$ was arbitrary. This completes the proof of the proposition.   
\end{proof}

Now, combining Propositions~\ref{prop:main_lqg_estimate} and ~\ref{prop:tightness}, we prove in Proposition~\ref{prop:conv_in_prob} below an analog of Theorem~\ref{thm:main_result} with the metrics $(\widetilde{D}_h^{\xi_n})_{n \in \NN}$ in place of $(\widehat{D}_h^{\xi_n})_{n \in \NN}$.

\begin{proposition}\label{prop:conv_in_prob}
Suppose that the functions $b,M$ are chosen as in Proposition~\ref{prop:tightness}. Let $(\xi_n)_{n \in \NN}$ be a sequence in $(0,\xi_c)$ such that $\xi_n \to \xi_c$ as $n \to \infty$. Then, the sequence of metrics $(\widetilde{D}_h^{\xi_n})_{n \in \NN}$ is tight with respect to the local uniform topology on $\CC \times \CC$. Moreover, if $(\xi_{k_n})_{n \in \NN}$ is a subsequence of $(\xi_n)_{n \in \NN}$ such that $(\widetilde{D}_h^{\xi_{k_n}})_{n \in \NN}$ converges in law (with respect to the local uniform topology on $\CC \times \CC$) to some random metric on $\CC \times \CC$, then we have that the following is true. There exists a $\xi_c$-LQG metric $h \mapsto \widetilde{D}_h$ and a subsequence $(\xi_{k'_n})_{n \in \NN}$ of $(\xi_{k_n})_{n \in \NN}$ such that
\begin{equation*}
    \widetilde{D}_h^{\xi_{k'_n}} \to \widetilde{D}_h \quad \text{in probability} \quad \text{as} \quad n \to \infty.
\end{equation*}    
\end{proposition}

\begin{proof}
First, we note that Proposition~\ref{prop:tightness} (applied with $U = \CC$) implies that the sequence of random metrics $(\widetilde{D}_h^{\xi_n})_{n \in \NN}$ is tight with respect to the local uniform topology of functions on $\CC \times \CC$, whenever $h$ is a whole-plane GFF normalized such that $h_1(0) = 0$. 

Let $(\xi_{k_n})_{n \in \NN}$ be a subsequence of $(\xi_n)_{n \in \NN}$ such that $\widetilde{D}_h^{\xi_{k_n}}$ converges in law as $n \to \infty$ (with respect to the local uniform topology on $\CC \times \CC$) to some random metric $D$. Using the Skorokhod representation theorem (by possibly passing into a subsequence of $(\xi_{k_n})_{n \in \NN}$), we can couple the metric $D$ with a collection of random fields $\{h^n , h\}_{n \in \NN}$ such that for all $n \in \NN$, both $h^n$ and $h$ have the law of a whole-plane GFF normalized so that its average on $\partial B_1(0)$ is equal to zero, and such that
\begin{equation*}
    (h^n , \widetilde{D}_{h^n}^{\xi_{k_n}}) \to (h,D) \quad \text{as} \quad n \to \infty,\,\,\text{a.s.},
\end{equation*}
where the convergence in the first coordinate is considered with respect to $H_{\text{loc}}^{-1}(\CC)$ and the convergence in the second coordinate is considered with respect to the local uniform topology of functions on $\CC \times \CC$.

Note that $\widetilde{D}_{h^n}^{\xi_{k_n}}$ also converges to $D$ as $n \to \infty$ a.s. with respect to the lower semicontinuous topology of functions induced by the convergence in Definition~\ref{def:lower_semi_cont}. Clearly, $D$ satisfies the triangle inequality a.s. In particular, the random metrics $(h^n, \widetilde{D}_{h^n}^{\xi_{k_n}})_{n \in \NN}, (h,D)$ are coupled as in \cite[Proposition~7.3]{kavvadias2025continuity}. Therefore, it follows from the proof of \cite[Theorem~8.19]{kavvadias2025continuity} that there exists a $\xi_c$-LQG metric $\psi \mapsto \widetilde{D}_{\psi}$ such that $D = \widetilde{D}_{h}$ a.s. This implies that $D$ is a.s. determined by $h$. Hence, combining with \cite[Lemma~10.2]{kavvadias2025continuity} (see also \cite[Lemma~4.5]{SS13}) and since $\widetilde{D}_{h^n}^{\xi_{k_n}}$ is a.s. determined by $h^n$ for all $n \in \NN$, we obtain that 
\begin{equation*}
    \widetilde{D}_h^{\xi_{k_n}} \to \widetilde{D}_h \quad \text{in probability} \quad \text{as} \quad n \to \infty
\end{equation*}
with respect to the local uniform topology on $\CC \times \CC$. This completes the proof of the proposition.
\end{proof}

Finally, we are ready to complete the proofs of Theorem~\ref{thm:main_result} and Corollary~\ref{cor:main_result}. Theorem~\ref{thm:main_result} will follow from combining the uniqueness of LQG metrics (Theorem~\ref{thm:uniqueness_lqg_metrics}) with Proposition~\ref{prop:conv_in_prob}. Corollary~\ref{cor:main_result} will follow from an argument which is similar to the argument given in the proof of \cite[Lemma~7.1]{dubedat2020liouville}.

\begin{proof}[Proof of Theorem~\ref{thm:main_result}.]
Suppose that we have the same setup as in Proposition~\ref{prop:conv_in_prob}. Let $(\xi_{k_n})_{n \in \NN}$ be any subsequence of $(\xi_n)_{n \in \NN}$. Then, Proposition~\ref{prop:conv_in_prob} implies that there exists a further subsequence $(\xi_{k'_n})_{n \in \NN}$ of $(\xi_{k_n})_{n \in \NN}$ such that
\begin{equation*}
    \widetilde{D}_h^{\xi_{k'_n}} \to \widetilde{D}_h \quad \text{in probability} \quad \text{as} \quad n \to \infty
\end{equation*}
for some $\xi_c$-LQG metric $\psi \mapsto \widetilde{D}_{\psi}$. By possibly passing into a subsequence, we can assume that
\begin{equation*}
    \widetilde{D}_h^{\xi_{k'_n}} \to \widetilde{D}_h \quad \text{as} \quad n \to \infty \quad \text{a.s.}
\end{equation*}
with respect to the local uniform topology of functions on $\CC \times \CC$.

Let $\widetilde{\beta}(\xi_{k_n}) , \widetilde{\beta}(\xi_c) \in (0,\infty)$ be defined by
\begin{align*}
    \widetilde{\beta}(\xi_{k_n}):=\inf\{l > 0 : \mathbb{P}[\widetilde{D}_h^{\xi_{k_n}}(0,1) \leq l] > 1/2 \},\,\, \widetilde{\beta}(\xi_c):=\inf\{l>0 : \mathbb{P}[\widetilde{D}_h(0,1) \leq l] > 1/2\},
\end{align*}
and note that Lemma~\ref{lem:normalization_well_defined} combined with Theorem~\ref{thm:uniqueness_lqg_metrics} imply that
\begin{equation*}
\mathbb{P}[\widetilde{D}_h^{\xi_{k_n}}(0,1) \leq \widetilde{\beta}(\xi_{k_n})] = \mathbb{P}[\widetilde{D}_h(0,1) \leq \widetilde{\beta}(\xi_c)] = \frac{1}{2}.  
\end{equation*}
Since $\widetilde{D}_h^{\xi_{k'_n}}(0,1) \to \widetilde{D}_h(0,1)$ as $n \to \infty$ a.s., we obtain by combining with Lemma~\ref{lem:normalization_well_defined} and Theorem~\ref{thm:uniqueness_lqg_metrics} that $\widetilde{\beta}(\xi_{k'_n}) \to \widetilde{\beta}(\xi_c)$ as $n \to \infty$. Thus, it follows that
\begin{equation*}
    \widetilde{\beta}(\xi_{k'_n})^{-1} \widetilde{D}_h^{\xi_{k'_n}} \to \widetilde{\beta}(\xi_c)^{-1} \widetilde{D}_h \quad \text{in probability} \quad \text{as} \quad n \to \infty
\end{equation*}
with respect to the local uniform topology of functions on $\CC \times \CC$.

Finally, we note that 
\begin{equation*}
\widehat{D}_h^{\xi_{k'_n}} = \widetilde{\beta}(\xi_{k'_n})^{-1} \widetilde{D}_h^{\xi_{k'_n}},\,\,\widehat{D}_h^{\xi_c} = \widetilde{\beta}(\xi_c)^{-1} \widetilde{D}_h \quad \text{for all} \quad n \in \NN \quad \text{a.s.}    
\end{equation*}
by the uniqueness of LQG metric modulo a multiplicative constant (see Theorem~\ref{thm:uniqueness_lqg_metrics}) combined with Lemma~\ref{lem:normalization_well_defined}. Therefore, it follows that
\begin{equation*}
\widehat{D}_h^{\xi_{k'_n}} \to \widehat{D}_h^{\xi_c} \quad \text{in probability} \quad \text{as} \quad n \to \infty.   
\end{equation*}
The proof of the theorem is then complete since the subsequence $(\xi_{k_n})_{n \in \NN}$ of $(\xi_n)_{n \in \NN}$ was arbitrary.    
\end{proof}

\begin{proof}[Proof of Corollary~\ref{cor:main_result}.]
\emph{Step 1. Outline of the proof.} By possibly passing into a subsequence of $(\xi_n)_{n \in \NN}$ and applying Theorem~\ref{thm:main_result}, we can assume that it is a.s. the case that $\widehat{D}_h^{\xi_n}$ converges to $\widehat{D}_h^{\xi_c}$ as $n \to \infty$ locally uniformly.

In Step 2, we will show that a.s., the sequence of metrics $(\widehat{D}_{h+f_n}^{\xi_n})_{n \in \NN}$ is locally equicontinuous for any such choice of $(f_n)_{n \in \NN}$ and $f$. Hence, by the Arz\'ela-Ascoli theorem, to complete the proof of the corollary (and since the sequence $(\xi_n)_{n 
\in \mathbb{N}}$ was arbitrary), it suffices to show that
\begin{equation*}
    \widehat{D}_{h+f_n}^{\xi_n}(z,w) \to \widehat{D}_{h+f}^{\xi_c}(z,w) \,\,\text{as} \,\, n \to \infty \,\, \text{for all} \, \, (z,w) \in \CC \times \CC.
\end{equation*}

In Step 3, we will show that 
\begin{equation}\label{eq:6}
\limsup_{n \to \infty} \widehat{D}_{h+f_n}^{\xi_n}(z,w) \leq \widehat{D}_{h+f}^{\xi_c}(z,w) \quad \text{for all} \quad (z,w) \in \mathbb{C} \times \mathbb{C},
\end{equation}
and in Step 4, we will show that
\begin{equation}\label{eq:10}
\liminf_{n \to \infty} \widehat{D}_{h+f_n}^{\xi_n}(z,w) \geq \widehat{D}_{h+f}^{\xi_c}(z,w) \quad \text{for all} \quad (z,w) \in \mathbb{C} \times \mathbb{C}.
\end{equation}

Hence, we will complete the proof  by combining \eqref{eq:6} with \eqref{eq:10}.

\emph{Step 2. The sequence of random metrics $(\widehat{D}_{h+f_n}^{\xi_n})_{n \in \NN}$ is locally equicontinuous.}

First, we note that Weyl scaling (Axiom~\eqref{it:weyl_scaling}) implies that
\begin{equation*}
\exp(-\xi_n ||f_n||_{\infty}) \widehat{D}_h^{\xi_n}(z,w) \leq \widehat{D}_{h+f_n}^{\xi_n}(z,w) \leq \exp(\xi_n ||f_n||_{\infty}) \widehat{D}_h^{\xi_n}(z,w)\,\,\text{for all}\,\,n \in \NN,\,\,z,w \in \CC.
\end{equation*}
Since $\xi_n \to \xi_c$ as $n \to \infty$, we obtain that there exists (random) constant $C>0$ such that
\begin{align}\label{eq:1}
C^{-1} \widehat{D}_{h}^{\xi_n}(z,w) \leq \widehat{D}_{h+f_n}^{\xi_n}(z,w) \leq C \widehat{D}_h^{\xi_n}(z,w),\,\, C^{-1} \widehat{D}_h^{\xi_c}(z,w) \leq \widehat{D}_{h+f}^{\xi_c}(z,w) \leq C \widehat{D}_h^{\xi_c}(z,w)   
\end{align}
for all $n \in \NN, z,w \in \CC$.

Hence, \eqref{eq:1} implies that for all $m,n \in \NN, (z,w),(z',w') \in B_m(0) \times B_m(0)$, we have that
\begin{align}\label{eq:2}
|\widehat{D}_{h+f_n}^{\xi_n}(z,w) - \widehat{D}_{h+f_n}^{\xi_n}(z',w')|&\leq |\widehat{D}_{h+f_n}^{\xi_n}(z,w) - \widehat{D}_{h+f_n}^{\xi_n}(z',w)| + |\widehat{D}_{h+f_n}^{\xi_n}(z',w) - \widehat{D}_{h+f_n}^{\xi_n}(z',w')|\nonumber\\
&\leq \widehat{D}_{h+f_n}^{\xi_n}(z,z') + \widehat{D}_{h+f_n}^{\xi_n}(w,w')\nonumber\\
&\leq C \left(\widehat{D}_h^{\xi_n}(z,z') + \widehat{D}_h^{\xi_n}(w,w')\right).    
\end{align}
Moreover, since $\widehat{D}_h^{\xi_n} \to \widehat{D}_h^{\xi_c}$ as $n \to \infty$ locally uniformly on $\CC \times \CC$, we obtain by combining ~\eqref{eq:2} with \cite[Theorem~1.7]{ding2024critical} that $(\widehat{D}_{h+f_n}^{\xi_n})_{n \in \NN}$ is equicontinuous.

\emph{Step 3. Proof of \eqref{eq:6}.}
Fix $z, w \in \CC$ distinct points. We note that combining \cite[Theorem~1.7]{DFGPS20}, \cite[Proposition~5.19]{ding2023tightness}, and \cite[Theorem~1.7, Proposition~1.8]{ding2024critical} gives that it is a.s. the case that the length metric spaces $(\CC,\widehat{D}_h^{\xi_n})$ and $(\CC,\widehat{D}_h^{\xi_c})$ are complete and locally compact for all $n \in \NN$. Thus, \eqref{eq:1} implies that the same is true for the length metrics spaces $(\CC,\widehat{D}_{h+f_n}^{\xi_n})$ and $(\CC,\widehat{D}_{h+f}^{\xi_c})$ for all $n \in \NN$. Therefore, \cite[Theorem~2.5.23]{burago2001course} implies that it is a.s. the case that both $(\CC,\widehat{D}_{h+f_n}^{\xi_n})$ and $(\CC,\widehat{D}_{h+f}^{\xi_c})$ are geodesic metric spaces for all $n \in \NN$.

Since $\widehat{D}_h^{\xi_c}(B_m(0) , \partial B_k(0)) \to \infty$ as $k \to \infty$ for all $m \in \NN$ by \cite[Proposition~5.19]{ding2023tightness} and 
\begin{equation*}
    \sup_{x,y \in B_m(0)} \widehat{D}_h^{\xi_n}(x,y) \to \sup_{x,y \in B_m(0)} \widehat{D}_h^{\xi_c}(x,y) \quad \text{as}\quad n \to \infty,
\end{equation*}
we obtain that there exists $m' \in \NN, m'>m$ such that
\begin{equation*}
    \widehat{D}_h^{\xi_n}(B_m(0), \partial B_{m'}(0)) > \sup_{x,y \in B_m(0)} \widehat{D}_h^{\xi_n}(x,y) \quad \text{for all} \quad n \in \NN. 
\end{equation*}
Therefore, combining with \eqref{eq:1}, we obtain that we can assume that $m,m'\in \NN$ are such that $m<m', z,w \in B_m(0)$, and such that the following holds for all $n \in \NN, x,y \in B_m(0)$. Every $\widehat{D}_h^{\xi_n}$-geodesic (resp.\ $\widehat{D}_{h+f_n}^{\xi_n}$-geodesic) from $x$ to $y$ is contained in $B_{m'}(0)$ and every $\widehat{D}_h^{\xi_c}$-geodesic (resp.\ $\widehat{D}_{h+f}^{\xi_c}$-geodesic) from $x$ to $y$ is contained in $B_{m'}(0)$.

Fix $\varepsilon>0$. Then, combining \eqref{eq:1} with \cite[Proposition~1.8]{ding2024critical} and the local uniform convergence as $n \to \infty$ of $\widehat{D}_h^{\xi_n}$ to $\widehat{D}_h^{\xi_c}$ gives that it is a.s. the case that there exists (random) $\delta>0$ such that the following holds for all $n \in \NN$ and all $x,y \in B_{m'}(0)$ with $|x-y| < \delta$. Any $\widehat{D}_h^{\xi_n}$-geodesic (resp.\ $\widehat{D}_h^{\xi_c}$-geodesic) from $x$ to $y$ and any $\widehat{D}_{h+f_n}^{\xi_n}$-geodesic (resp.\ $\widehat{D}_{h+f}^{\xi_c}$-geodesic) from $x$ to $y$ is contained in $B_{\varepsilon}(x)$.

Let $z=z_0,\cdots,z_N=w \in B_{m'}(0)$ be such that $\widehat{D}_{h+f}^{\xi_c} = \sum_{i=0}^{N-1} \widehat{D}_{h+f}^{\xi_c}(z_i,z_{i+1})$ and $\widehat{D}_{h+f}^{\xi_c}(z_i,z_{i+1}) < \delta$ for all $i=0,1,\cdots,N-1$ (note that the choice of the $z_i$'s is possible since $\widehat{D}_{h+f}^{\xi_c}$ is a geodesic metric and every $\widehat{D}_{h+f}^{\xi_c}$-geodesic from $z$ to $w$ is contained in $B_{m'}(0)$). Note that
\begin{equation*}
e^{-\text{osc}(f_n , K_{x,y}^n)\xi_n} e^{\xi_n f_n(x)} \widehat{D}_h^{\xi_n}(x,y) \leq \widehat{D}_{h+f_n}^{\xi_n}(x,y)\leq e^{\text{osc}(f_n,K_{x,y}^n) \xi_n} e^{\xi_n f_n(x)} \widehat{D}_h^{\xi_n}(x,y),
\end{equation*}
where
\begin{equation*}
\text{osc}(f_n , K):=\sup_{x,y \in K} |f_n(x) - f_n(y)|   
\end{equation*}
and $K_{x,y}^n$ denotes the union between a $\widehat{D}_h^{\xi_n}$-geodesic from $x$ to $y$ and a $\widehat{D}_{h+f_n}^{\xi_n}$-geodesic from $x$ to $y$.

Note that for $n \in \NN$ large enough and for every $i=0,\cdots,N-1$, we have that $\widehat{D}_h^{\xi_n}(z_i,z_{i+1}) < \delta$. Hence, by the triangle inequality, we have for $n$ large enough that
\begin{align}\label{eq:4}
\widehat{D}_{h+f_n}^{\xi_n}(z,w)&\leq \sum_{i=0}^{N-1} \widehat{D}_{h+f_n}^{\xi_n}(z_i,z_{i+1})\nonumber\\
&\leq \sum_{i=0}^{N-1} e^{\text{osc}(f_n,K_{z_i,z_{i+1}}^n)\xi_n} e^{\xi_n f_n(z_i)} \widehat{D}_h^{\xi_n}(z_i,z_{i+1})\nonumber\\
&\leq e^{\omega(f_n,\varepsilon)\xi_n} \sum_{i=0}^{N-1} e^{\xi_n f_n(z_i)} \widehat{D}_h^{\xi_n}(z_i,z_{i+1}),   
\end{align}
where
\begin{equation*}
    \omega(g,\varepsilon):=\sup\{|g(x)-g(y)| : x,y \in B_{2m'}(0),|x-y| < \varepsilon\}
\end{equation*}
for a continuous function $g$. Hence, taking $n \to \infty$ in \eqref{eq:4} gives that
\begin{align}\label{eq:5}
\limsup_{n \to \infty} \widehat{D}_{h+f_n}^{\xi_n}(z,w) &\leq e^{\xi_c \omega(f,\varepsilon)} \sum_{i=0}^{N-1} e^{\xi_c f(z_i)} \widehat{D}_h^{\xi_c}(z_i,z_{i+1})\nonumber\\
&\leq e^{\xi_c \omega(f,\varepsilon)} \sum_{i=0}^{N-1} e^{\text{osc}(f,K_{z_i,z_{i+1}})\xi_c} \widehat{D}_{h+f}^{\xi_c}(z_i,z_{i+1})\nonumber\\
&\leq e^{2\xi_c \omega(f,\varepsilon)} \sum_{i=0}^{N-1} \widehat{D}_{h+f}^{\xi_c}(z_i,z_{i+1})\nonumber\\
&=e^{2\xi_c \omega(f,\varepsilon)} \widehat{D}_{h+f}^{\xi_c}(z,w),    
\end{align}
where $K_{z_i,z_{i+1}}$ is defined in the exact same way as $K_{z_i,z_{i+1}}^n$ but with $\widehat{D}_h^{\xi_c}$ and $\widehat{D}_{h+f}^{\xi_c}$ in place of $\widehat{D}_h^{\xi_n}$ and $\widehat{D}_{h+f_n}^{\xi_n}$ respectively. By letting $\varepsilon \to 0$ in \eqref{eq:5} and using the local uniform continuity of $f$, we obtain \eqref{eq:6}.

\emph{Step 4. Proof of \eqref{eq:10}.}
Next, we let $z=z_0^n,\cdots,z_{N_n}^n = w \in B_{m'}(0)$ be such that
\begin{align}\label{eqn:6}
&\widehat{D}_{h+f_n}^{\xi_n}(z,w) = \sum_{i=0}^{N_n-1} \widehat{D}_{h+f_n}^{\xi_n}(z_i^n,z_{i+1}^n),\nonumber\\
&\widehat{D}_h^{\xi_n}(z_i^n,z_{i+1}^n) <\delta,\quad \text{for all} \quad n \in \NN, i=0,\cdots,N_n-1.   
\end{align}

Taking the minimal number $N_n$ for which \eqref{eqn:6} holds, we can assume by possibly taking a subsequence that $N_n$ converges as $n \to \infty$ to some number $N \in \NN$. Moreover, by possibly passing into a further subsequence, we can assume that
\begin{equation*}
    \lim_{n \to \infty} \widehat{D}_{h+f_n}^{\xi_n}(z,w) = \liminf_{n \to \infty} \widehat{D}_{h+f_n}^{\xi_n}(z,w)
\end{equation*}
and the $z_i^n$'s converge to some $z_i$'s for $0 \leq i \leq N$ and $\widehat{D}_h^{\xi_c}(z_i,z_{i+1}) \leq\delta$ for all $0 \leq i \leq N$. Then, for $n$ large enough, we have that
\begin{align}\label{eq:7}
\widehat{D}_{h+f_n}^{\xi_n}(z,w)&\geq \sum_{i=0}^{N-1} e^{-\xi_n \text{osc}(f_n , K_{z_i^n,z_{i+1}^n}^n)} e^{\xi_n f_n(z_i^n)} \widehat{D}_h^{\xi_n}(z_i^n,z_{i+1}^n)\nonumber\\
&\geq e^{-\xi_n \omega(f_n,\varepsilon)} \sum_{i=0}^{N-1} e^{\xi_n f_n(z_i^n)} \widehat{D}_h^{\xi_n}(z_i^n,z_{i+1}^n).   
\end{align}

Using the local uniform convergence of $\widehat{D}_h^{\xi_n}$ to $\widehat{D}_h^{\xi_c}$, we obtain that as $n \to \infty$,
\begin{align}\label{eq:8}
&\left |\sum_{i=0}^{N-1} e^{\xi_n f_n(z_i^n)} \widehat{D}_h^{\xi_n}(z_i^n,z_{i+1}^n) - \sum_{i=0}^{N-1} e^{\xi_n f_n(z_i^n)} \widehat{D}_h^{\xi_c}(z_i^n,z_{i+1}^n)\right |\nonumber\\
&\leq N e^{\xi_n ||f_n||_{\infty}} ||\widehat{D}_h^{\xi_c}|_{B_{2m'}(0) \times B_{2m'}(0)} - \widehat{D}_h^{\xi_n}|_{B_{2m'}(0) \times B_{2m'}(0)}||_{\infty} \to 0.   
\end{align}

Combining \eqref{eq:7} with \eqref{eq:8} gives that
\begin{align}\label{eq:9}
\liminf_{n \to \infty} \widehat{D}_{h+f_n}^{\xi_n}(z,w) &\geq e^{-\xi_c \omega(f,\varepsilon)} \sum_{i=0}^{N-1} e^{\xi_c f(z_i)} \widehat{D}_h^{\xi_c}(z_i,z_{i+1})\nonumber\\
&\geq e^{-\xi_c \omega(f,\varepsilon)} \sum_{i=0}^{N-1} e^{-\xi_c \omega(f,K_{z_i,z_{i+1}})} \widehat{D}_{h+f}^{\xi_c}(z_i,z_{i+1})\nonumber\\
&\geq e^{-2\xi_c \omega(f,\varepsilon)} \sum_{i=0}^{N-1} \widehat{D}_{h+f}^{\xi_c}(z_i,z_{i+1})\nonumber\\
&\geq e^{-2\xi_c \omega(f,\varepsilon)} \widehat{D}_{h+f}^{\xi_c}(z,w).   
\end{align}

Letting $\varepsilon \to 0$ in \eqref{eq:9} gives \eqref{eq:10}.
Therefore, the proof of the corollary is complete.  
\end{proof}

\printbibliography
\end{document}